\documentclass{article}

\usepackage[utf8]{inputenc}
\usepackage[english]{babel}

\usepackage{amsmath,amsthm,amssymb,mathrsfs,yhmath,stmaryrd,colonequals}

\usepackage{mathtools}
\usepackage{amsfonts}
\usepackage{enumerate}
\usepackage{enumitem}
\usepackage{hyperref}

\usepackage{tikz}
\usetikzlibrary{cd}

\allowdisplaybreaks
\numberwithin{equation}{section}

\newcommand{\Z}{\mathbb{Z}}
\newcommand{\Q}{\mathbb{Q}}
\newcommand{\R}{\mathbb{R}}
\newcommand{\C}{\mathbb{C}}
\newcommand{\Hq}{\mathbb{H}}

\newcommand{\Pj}{\mathbf{P}}
\newcommand{\Sph}{\mathbf{S}}
\newcommand{\Hy}{\mathbf{H}}

\newcommand{\I}{\mathrm{i}}

\newcommand{\U}{\mathrm{U}}
\newcommand{\SU}{\mathrm{SU}}

\newcommand{\SP}{\mathrm{Sp}}

\newcommand{\Hom}{\mathrm{Hom}}

\newcommand{\Alt}{\mathrm{Alt}}
\newcommand{\Sym}{\mathrm{Sym}}
\newcommand{\sign}{\mathrm{sign}}

\newcommand{\tw}{\mathrm{tw}}
\newcommand{\vt}{\mathrm{vert}}
\newcommand{\ol}{\overline}
\newcommand{\wt}{\widetilde}
\newcommand{\smooth}[1]{\mathcal{C}^{\infty}\!{({#1})}}
\newcommand{\field}[1]{\mathfrak{X}{\left({{#1}} \right)}}
\newcommand{\spn}[1]{\left\langle {#1} \right\rangle}
\newcommand{\id}{\mathrm{id}}

\newcommand{\Lie}[1]{\mathcal{L}_{#1}}

\newcommand{\dif}{\mathrm{d}}
\newcommand{\im}{\mathrm{Im}}
\newcommand{\rk}{\mathrm{rk}}

\renewcommand{\Re}{\operatorname{Re}}
\renewcommand{\Im}{\operatorname{Im}}

\newcommand{\norm}[1]{\left\lVert{#1}\right\rVert}

\theoremstyle{definition}

\newtheorem{defi}{Definition}[section]
\newtheorem{eg}[defi]{Example}

\theoremstyle{plain}

\newtheorem{theo}[defi]{Theorem}
\newtheorem{prop}[defi]{Proposition}
\newtheorem{cor}[defi]{Corollary}
\newtheorem{lemma}[defi]{Lemma}

\newtheorem*{theo*}{Theorem}

\theoremstyle{remark}
\newtheorem{rmk}[defi]{Remark}

\title{The c-map as a functor on certain\\ variations of Hodge structure}
\author{Mauro Mantegazza\footnote{Masaryk University, Kotl\'{a}\v{r}sk\'{a} 2, Bldg.\ 8, 602 00, Brno, Czech Republic.\newline E-mail: \mbox{\href{mailto:mauro.mantegazza.uni@gmail.com}{mauro.mantegazza.uni@gmail.com}}, ORCID: 0000-0001-7938-2950}, Arpan Saha\footnote{Universit\"{a}t Hamburg, Bundesstra{\ss}e 55, 20146 Hamburg, Germany. \newline E-mail: \href{mailto:arpan.saha@uni-hamburg.de}{arpan.saha@uni-hamburg.de}, ORCID: 0000-0003-4336-4229}}
\begin{document}
	\maketitle

\begin{abstract}
	We give a new manifestly natural presentation of the supergravity c-map.
	We achieve this by giving a more explicit description of the correspondence between projective special K\"ahler manifolds and variations of Hodge structure, and by demonstrating that the twist construction of Swann, for a certain kind of twist data, reduces to a quotient by a discrete group.
	We combine these two ideas by showing that variations of Hodge structure give rise to the aforementioned kind of twist data and by then applying the twist realisation of the c-map due to Macia and Swann.
	This extends previous results regarding the lifting of general isomorphisms along the undeformed c-map, and of infinitesimal automorphisms along the deformed c-map.
	We show in fact that general isomorphisms can be naturally lifted along the deformed c-map.
\end{abstract}
	
\section{Introduction}

\subsection{Background}

Quaternionic Kähler manifolds are notable as their holonomy group $\SP(n)\SP(1)$ is the only entry in Berger's classification of Riemannian holonomy groups that corresponds to Einstein but not Ricci-flat metrics. 
The construction of (complete) Einstein metrics has been a question of long-standing interest among geometers. 
The rigidity offered by the special holonomy of the quaternionic Kähler manifolds allows one to import tools from various other areas of mathematics, such as representation theory and algebraic geometry, to construct interesting examples of Einstein manifolds. 

This paper is concerned with one such way of constructing families of complete inhomogeneous quaternionic Kähler metrics with negative scalar curvature that has its origin in the literature on supergravity and string theory. The negativity of the scalar curvature is actually crucial, as there is considerable evidence for a conjecture due to LeBrun and Salamon \cite{LSConj} that there are no non-symmetric examples of complete quaternionic Kähler manifolds with positive scalar curvature. For example, this was shown to be the case in dimension $4$ in \cite{Hit81, FK82}, in dimension $8$ in \cite{PS91}, and in dimensions $12$ and $16$ in \cite{BWW}. Moreover, it was shown in \cite{DS99} that there are no cohomogeneity $1$ examples in any dimension. In fact, LeBrun and Salamon themselves showed in \cite{LSConj} that there are, up to rescaling, only finitely many examples of complete quaternionic Kähler manifolds with positive scalar curvature in each dimension.

In physics, quaternionic Kähler manifolds with negative scalar curvature can be thought of as manifolds parametrising the scalar fields in a $3$-dimensional supergravity theory defined on a ``spacetime'' manifold $X$ of dimension $3$. One way to produce such a theory is by \emph{dimensional reduction}. Briefly, one starts with an appropriate $4$-dimensional supergravity theory on $X \times \Sph^1$ and then takes a limit in which the radius of the circle $\Sph^1$ goes to zero. The manifold parametrising vector multiplet scalar fields in the $4$-dimensional supergravity theory has a natural \emph{projective special Kähler structure}, a notion that will be recalled later on. Thus, the result of this physical procedure is a map from the set of projective special Kähler manifolds to the set of quaternionic Kähler manifolds. This is the \emph{undeformed c-map}.

In the case of supergravity theories arising as low-energy descriptions of string theory, perturbative quantum corrections arising from the genus expansion of string theory give rise to a deformation of the above quaternionic Kähler metric parametrised by a discrete integer parameter $k$. When $k\neq 0$, the resulting map from the set of projective special Kähler manifolds to the set of families of quaternionic Kähler manifolds $N_{2k}$ is the \emph{deformed c-map}. These two cases will be collectively referred to as the \emph{supergravity c-map}. The adjective ``supergravity'' is to distinguish it from a similar construction in the absence of gravity called the \emph{rigid} c-map, that we shall also be encountering later.

The supergravity c-map was described in geometric terms in \cite{CiSC} and \cite{CortesCompProj} as a bundle construction over the projective special Kähler manifold.
Note that, in these works, the deformation parameter was continuous, as the authors work only up to local isometry.
In fact, up to local isometry, there are only three cases, characterised by the sign of the parameter.
When the given projective special Kähler metric is complete, it was proven in \cite{CortesCompProj}, that the resulting quaternionic Kähler metric is also complete for $k\ge 0$. This is in fact false for $k<0$ (see \cite[Proposition 4 (p.\ 287)]{CortesHKQK}), and so we henceforth exclude this case from our discussion of the supergravity c-map.

Projective special Kähler manifolds and (families of) quaternionic Kähler manifolds have natural notions of isomorphism, giving rise to groupoids (i.e.\ categories with only invertible arrows) of projective special K\"ahler and (families of) quaternionic K\"ahler manifolds.
So, one may ask whether the map of sets underlying the supergravity c-map actually descends to a map of isomorphism classes.
It was shown in \cite{CortesCompProj} that this is indeed the case.
However, the construction was still formulated in terms of non-canonical data, so it is not clear from this approach whether the supergravity c-map is a functor between the relevant groupoids.
In other words, what we are asking is whether isomorphisms of projective special Kähler manifolds \emph{canonically} lift to isomorphisms (i.e.\ isometries \cite[(p.\ 529)]{FlowsQKVSR}) of quaternionic Kähler manifolds.

Partial results had been previously established.
It had been shown by independent methods in \cite{CiSC}, \cite{MaciaSwann2015} and \cite{MaciaSwann2019} that in the case of the undeformed c-map, the quaternionic Kähler manifold forms a trivial bundle over the projective special Kähler manifold.
In particular, the lifting of automorphisms is well-defined \cite[Appendix]{CortesCohomogeneityOne}.
In the case of the deformed c-map, it was shown in \cite{CST} that there is a well-defined lifting of infinitesimal automorphisms.
The question of whether general automorphisms and isomorphisms naturally lift in the deformed case hitherto remained open.

\subsection{Main results}

In this paper, we fill in the gap and prove that isomorphisms of projective special Kähler manifolds canonically lift to global isometries of the quaternionic Kähler manifolds under the supergravity c-map, for non-negative integer deformation parameters.

There are two main ingredients in our approach. The first is a reformulation of the notion of projective special Kähler structure as a particular case of (abstract) variations of Hodge structure of weight $3$.
We will define these in greater detail later, but for now, we will just mention that they are an abstraction of the bundle $E\rightarrow M$, whose base $M$ is the moduli space of complex structure deformations of a compact simply connected Calabi--Yau $3$-fold $K_{m\in M}$, and whose fibres $E_m$ are the cohomology groups $H^3_{\mathrm{dR}}(K_m)$.
Accordingly, we have a Hodge decomposition 
\begin{equation}
	E_{\C} = E^{3,0} \oplus E^{2,1} \oplus E^{1,2} \oplus E^{0,3}.
\end{equation} 
In the case we are interested in, $E^{3,0}$ will be a holomorphic line bundle $L$ with hermitian structure $h_L$, and $E^{2,1}$ the bundle $L\otimes T^{1,0}M$.
The other two summands $E^{1,2}$ and $E^{0,3}$ are defined to be the conjugates of $E^{2,1}$ and $E^{3,0}$ respectively.
Moreover, $E^{3,0}$ has a canonical $\U(1)$-action that preserves $h_L$.

The correspondence between projective special Kähler manifolds and abstract variations of Hodge structure was first noted and proved by Cortés in \cite{CortesHK1998}.
A more intrinsic formulation was then given by Freed in \cite{Freed1999}.
We however give slightly different statements and proofs in Propositions \ref{pr:VPHS1} and \ref{pr:VPHS2} that describe the correspondence more explicitly.
This correspondence was also addressed in \cite{HHP} via a more local approach.
Our approach differs from theirs in that we give a global description of this structure, and in terms of a Hodge decomposition rather than a filtration.
The result is that we get a much clearer description of the relation, including an explicit expression for the Gau\ss--Manin connection.

The other ingredient is due to Macia and Swann \cite{MaciaSwann2015,MaciaSwann2019}, building on a prior work by Haydys \cite{Haydys}.
In their work, the c-map is interpreted as a special case of a more general construction called the \emph{twist}.
Introduced by Swann in \cite{Swann2010}, the twist construction takes as its input a manifold $S$ with certain \emph{twist data}, that includes a $\U(1)$-action, and produces as output another manifold $S'$ with a $\U(1)$-action.
In Lemma \ref{lemma:twist}, we establish that for certain choices of twist data, the twist construction is equivalent to taking a quotient by a finite subgroup $\Z_{k}$ of the $\U(1)$ acting on $S$.
This observation is new.

The description of projective special Kähler structures in terms of variations of Hodge structure turns out to be particularly well-suited for applying the twist construction.
Putting these ingredients together, we obtain the following result.

\begin{theo*}
	Let $E$ be the variation of Hodge structure associated to a projective special Kähler manifold $M$ and let $\wt M_{>2k}$ denote the subset of the line bundle $L\colonequals E^{3,0}$ of points $u$ whose fibrewise $h_L$-norm satisfies $h_L(u,u)>2k$. Then, the quaternionic manifold $N_{2k}$ given by the deformed c-map is the pullback of the real part of $E$ to the quotient $\wt M_{>2k}/\Z_k$, where $\Z_k$ is a subgroup of the $\U(1)$ whose action preserves $h_L$.
\end{theo*}

In the above, we have omitted the explicit expression for the quaternionic Kähler metric as it requires a large number of prerequisites. The complete statement is given and proved in Theorem \ref{th:sugra-c} later on.

Presented in this form, the deformed c-map construction is manifestly functorial, and elucidates the isometric action on the resulting quaternionic K\"ahler manifold by a semidirect product of the group of automorphisms of the projective special K\"ahler manifold with a Heisenberg group, cf.\ Remark \ref{rmk:semidirectHeis}.
As it turns out, the functorial lift of automorphisms of projective special Kähler manifolds reproduces the lift of infinitesimal automorphisms constructed in a very different way in \cite{CST}.
We prove this in Proposition \ref{prop:vorsicht-funktor}.
Furthermore, in Remark \ref{rmk:k_integer}, we provide evidence that our method for canonically lifting isomorphisms through the c-map cannot be extended to non-integer $k$.

Aside from the manifest functoriality, what is also notable about the above description of the supergravity c-map is that it suggests certain natural generalisations applicable to a much wider class of variations of Hodge structure.
The resulting manifold is not quaternionic Kähler in general, as its real dimension may not even be divisible by $4$.
Interpreting these resulting manifolds is a question we leave for future investigation.

\subsection{Outline}
Since our major motivation is to prove the naturality of certain constructions, we need to set up some notation.
Thus, we start with a detailed preliminary discussion in \S \ref{sec:prelim}.
Here, we introduce our notation and recollect certain elementary constructions in differential geometry that we make use of later.

Following this, in \S \ref{sec:SK}, we set up the basic definitions of projective special Kähler geometry, establish useful formulas regarding them, introduce variations of Hodge structure, and realise projective special Kähler structures as certain instances of them. 

Next, in \S \ref{sec:twist}, we introduce Swann's twist construction and prove a useful lemma regarding it.

In \S \ref{sec:c-map}, we finally put all of this together to give a new description of the supergravity c-map and discuss some of its properties, such as functoriality and the existence of a Heisenberg action on the quaternionic K\"ahler manifold.

\subsection*{Acknowledgments}

This work was carried out at Mathematisches Forschungsinstitut Oberwolfach (MFO) as part of an Oberwolfach Leibniz Fellowship (``From prepotential to deviance'').
We are immensely grateful to the MFO for the hospitality and for providing us with such an excellent environment for research.
In the later stages of this paper, A.S.\ was supported by the DFG Emmy Noether grant AL 1407/2\nobreakdash-1, and M.M.\ by the GACR grant GA19-06357S.
The collaboration that lead to this paper was initiated in Hamburg where M.M.\ was supported by German Science Foundation (DFG) under Germany’s Excellence Strategy (EXC 2121 ``Quantum Universe''-390833306).

We also thank Vicente Cortés, Diego Conti, and Danu Thung for their comments and suggestions regarding an earlier draft of this paper.
Finally, we thank the anonymous reviewer for carefully reading through this work and for several helpful suggestions, which have improved the presentation of this article.
\subsection*{Data availability}
Data sharing is not applicable to this article as no new data were created or analysed in this study.
\section{Preliminaries, notation, and conventions}\label{sec:prelim}
We begin by first recounting some elementary notions from differential geometry and fixing our notation.

\subsection{Complex vector spaces}
Let $W$ be a real vector space.
We denote the complexification $W \otimes_\R \C$ as $W_\C$.
If $W$ is endowed with a \emph{complex structure}, i.e.\ $I\in \Hom_{\R}(W, W)$ such that $I^2=-\id_{W}$, then we have the decomposition $W_{\C}=W^{1,0}\oplus W^{0,1}$ into eigenspaces corresponding to the eigenvalues $\I$ and $-\I$ respectively.
Similarly, we have a decomposition $W^*_{\C}=W^*_{1,0}\oplus W^*_{0,1}$ given by the transpose of $I$.
We define the space $W^*_{p,q}$ to be the subspace of $(W^*_\C)^{\otimes (p+q)}$ built out of $p$ copies of $W^*_{1,0}$ and $q$ copies of $W^*_{0,1}$.
An element of $W^*_{p,q}$ will be said to be of type $(p,q)$.

Now let $V$ be a complex vector space. Following \cite{RedBook}, we denote by $V_\R$ its underlying real vector space and by $[V]$ its real part with respect to an antilinear involution $\kappa$, also known as \emph{real structure}, i.e.\
\begin{equation}
[V]\colonequals \{v\in V\mid\kappa(v)=v\}.
\end{equation}
This is a real vector space whose real dimension is equal to the complex dimension of $V$.

Notice that for any complex vector space $V$, not necessarily one equipped with a real structure, we may define another complex vector space $\overline{V}$, which has the same underlying real vector space $V_\R$, but on which the complex scalar multiplication $\C \times V_\R \rightarrow V_\R$ is given by
\begin{equation}
	(\lambda, v_\R) \longmapsto (\overline \lambda v)_\R,
\end{equation}
where $v_\R \in V_\R$ denotes the vector $v \in V$ regarded as a vector in $V_\R$.
 
The identity map on $V_\R$ then becomes an antilinear map $V \rightarrow \overline V$, which in turn gives a canonical real structure on $V \oplus \overline V$.
In terms of this real structure, we may then define
\begin{equation}
\llbracket V\rrbracket \colonequals [V\oplus \overline{V}].
\end{equation}

Putting all of these together, we have the canonical isomorphisms
\begin{equation}
[V]_{\C}\cong V,\qquad
\llbracket V\rrbracket _{\C}\cong V\oplus \overline{V},
 \qquad \llbracket V\rrbracket \cong V_\R.
\end{equation}
The first two are isomorphisms of complex vector spaces, whereas the last one is an isomorphism of real vector spaces.

The same notation will be used accordingly for vector bundles.

\subsection{Bundles}
The data of a fibre bundle consists of a total space $A$, a base space $B$, and a submersion $A\rightarrow B$.
When there is no ambiguity, or unless otherwise specified, we adopt the notation $p^A_B$ for the submersion that is part of the fibre bundle data.
We denote by $A_f$ the fibred product $A \times_B B'$, when interpreted as the pullback bundle of $A\to B$ along a map $f\colon B'\rightarrow B$, i.e.\
\begin{equation}
	\begin{tikzcd}
	A_f\arrow[r]\arrow[d]& A\arrow[d] \\B'\arrow[r, "f" ]& B
	\end{tikzcd}
\end{equation}
Although it is customary to use $f^*$ to denote both the pullback, as well as the codifferential, i.e.\ the transpose of the differential $f_*\colon TB' \rightarrow (TB)_f$, we shall reserve $f^*$ only for the latter in order to avoid confusion between the two.

A further motivation for the notation $A_f$ for the pullback of $A$ along $f$ is the following observation.
By interpreting points $u \in B$ as maps $u\colon 1\to B$ from the one-point space $1$, we can regard the fibre of $A$ over a point $u$ as the pullback of $A$ along the map $u$. In this case, the pullback notation $A_u$ agrees with the standard notation for the fibre of $A$ over $u$.

In case the map $f$ is a generic bundle submersion, we will simplify the notation by denoting the pullback of $A\to B$ along $p^{C}_{B}$ by $A_C$, when $B$ is clear from the context.

If $s\colon B \rightarrow A$ is a section of the fibre bundle $A \rightarrow B$, then there is a unique section $s_f\colon B' \rightarrow A_f$ of the fibre bundle $A_f \rightarrow B'$ such that the following diagram commutes.
\begin{equation}
\begin{tikzcd}
A_f\arrow[r]& A\\B'\arrow[r, "f" ]\arrow[u, "s_f"]& B\arrow[u, "s"']
\end{tikzcd}
\end{equation}
The section $s_f$ is called the \emph{pullback section} of $s$ (along $f$).
Again, for a point $u\colon 1\to B$, the pullback notation $s_u\colon 1\to A_u$ is coherent with the standard notation for the value of $s$ at $u$.

For the pullback of a section $s$ of $A\to B$ along $p^C_B$, we will simplify the notation by writing $s_C$, as we did for bundles.
We shall often denote the pullback section $s_f$ associated to $s$ as simply $s$ when the context makes it clear what $f$ is. So, in particular, if $\alpha$ is a section of $T^*B$, then $f^*\alpha$ is to be regarded as $f^*\colon (T^*B)_f \rightarrow T^*B'$ applied to $\alpha_f$. In fact, we shall often omit the $f^*$ in $f^*\alpha$ altogether when there is no possibility of confusion.

The pullback $A_A \rightarrow A$ of a bundle over itself will be referred to as the \emph{tautological pullback bundle} associated to $A$. This has a distinguished section called the \emph{tautological section} that we shall be denoting $\Phi_A$. This corresponds to the map $A\to A_A$ obtained by the pullback universal property on the square
\begin{equation}
	\begin{tikzcd}
	A\arrow[r, "\id_A"]\arrow[d,"\id_A"']& A\arrow[d, "p^A_B"] \\A\arrow[r, "p^A_B" ]& B
	\end{tikzcd}
\end{equation}
Equivalently, it can be defined by the property that for any section $s\colon B \rightarrow A$, the pullback section $(\Phi_A)_s$ of $(A_A)_s =A_{p^A_B\circ s}=A$ is $s$ itself. 

The tautological section has another significant property. It is a general fact that if $A\rightarrow B$ and $A'\rightarrow B$ are two bundles over the same base, and $f_A\colon A\rightarrow A$ and $f_{A'}\colon A'\rightarrow A'$ are weak bundle maps covering the same base map $f_B\colon B\rightarrow B$, then there is a canonical map $f\colon A\times_B A'\rightarrow A\times_B A'$ which covers both $f_A$ and $f_{A'}$.
In other words, we can complete a commuting cube
\begin{equation}
\begin{tikzcd}[row sep=tiny, column sep=tiny]
& A\arrow[rrr,"f_A", pos=.4]\arrow[dd] & & & A\arrow[dd]\\
A\times_B A'\arrow[rrr, near end,"f", crossing over]\arrow[dd]\arrow[ru] & & & A\times_B A'\arrow[ru]\\
& B\arrow[rrr,"f_B", pos=.4] & & & B\\
A'\arrow[rrr,"f_{A'}"]\arrow[ru] & & & A'\arrow[ru]\arrow[<-,uu, crossing over]
\end{tikzcd}
\end{equation}
When $A'$ is taken to be equal to $A$, then this construction tells us that any weak bundle map $f_A$ induces a weak bundle map on the tautological pullback bundle $A_A$ covering $f_A$.
In an abuse of notation, we will denote this induced map in the same way as the map $f_A$ downstairs.
The tautological section $\Phi_A$ commutes with such induced maps.

In the case where $A\rightarrow B$ is a vector bundle, the tautological pullback bundle $A_A \rightarrow A$ is part of the following short exact sequence of vector bundles over $A$.
\begin{equation}\label{diag:exact}
\begin{tikzcd}
	0 \arrow[r] &A_A \arrow[r,"\vt"]& TA \arrow[r,"\left(p^A_B\right)_*\ "]& (TB)_A \arrow[r] &0
\end{tikzcd}
\end{equation}
The image of $\vt$ is the ``vertical'' subbundle $\mathcal V$ of $TA$, i.e.\ the subbundle whose sections are vector fields tangent to the fibres.

Although this sequence splits, the splitting is not canonical. In order to have a canonical splitting, we need a connection cf.\ \cite[Exercise 29 (p.\ 103)]{SpivakI}. 

\subsection{Connections}\label{ssec:Connections}

The most general notion of connection is an Ehresmann connection, which is applicable to any (smooth) fibre bundle.
This is just a choice of a right inverse to the map $(p^A_B)_*\colon TA\rightarrow (TB)_A$.
The image of such a right inverse is a ``horizontal'' subbundle $\mathcal H$ of $TA$, complementary to the vertical subbundle and isomorphic to $(TB)_A$.
In other words, the data of an Ehresmann connection is precisely the choice of a splitting $TA \cong \mathcal V \oplus \mathcal H$. 

The right inverse of $(p^A_B)_*\colon TA\rightarrow (TB)_A$ maps the pullback $X_A$ associated to a vector field $X$ on $B$, to a vector field $X^{\mathcal{H}}$ on the total space $A$.
We call $X^{\mathcal{H}}$ the \emph{horizontal lift} of $X$.

Morally, an Ehresmann connection gives us a way of parallel-transporting elements of a fibre of $A\rightarrow B$ to other fibres of the bundle.
Given a curve $\gamma\colon[0,1]\rightarrow B$, we have an induced map of bundles $\gamma'\colon A_{\gamma} \rightarrow A$.
The image of the differential $\gamma_*'\colon T(A_{\gamma})\rightarrow (TA)_{\gamma'}$ intersects $\mathcal H_{\gamma'}$ in a distribution that has rank at most $1$, and thus is integrable.
The integral curves of the distribution on $A_{\gamma}$ are said to be horizontal lifts of the curve $\gamma$.
Since each of these curves intersects every fibre of $A_{\gamma}$ at exactly one point, they provide a notion of parallel transport on the bundle $A\rightarrow B$ along $\gamma$.

Ehresmann connections reduce to the other notions of connection in the presence of additional structures (cf.\ \cite{KMS}). For instance, in the case of a $G$-invariant horizontal bundle $\mathcal H$ on a principal $G$-bundle, there is a unique $G$\nobreakdash-equivariant $\mathfrak{g}$-valued $1$-form $\varphi$ on the total space that vanishes on the subbundle $\mathcal H$ such that the following composition is the identity.
\begin{equation}
	\begin{tikzcd}
	\mathfrak{g} \arrow[r, "(\cdot)^{\circ}"] &\Gamma(\mathcal V) \arrow[r,"\varphi"]& \mathfrak{g}
	\end{tikzcd}
\end{equation}
Here $(\cdot)^\circ $ maps an element $X\in \mathfrak g$ to the corresponding fundamental vector field $X^\circ$ on $P$, defined by $X^{\circ}_u \colonequals \frac{\dif}{\dif t}(u\cdot\exp(tX))|_{t=0}$ at $u\in P$. Thus, a $G$-invariant Ehresmann connection on a principal $G$-bundle is equivalent to the choice of a principal connection $\varphi$.

Likewise, in the case of a vector bundle $A \rightarrow B$, the Ehresmann connection reduces to the data of a linear connection $\nabla\colon \Gamma(A) \rightarrow \Gamma(T^*B \otimes A)$ under certain conditions.
Let $s\in \Gamma(A)$ and $X\in\field{B}$, by the naturality of the Lie bracket, we deduce that $[X^{\mathcal{H}},\vt (s)]$ is a vertical vector field on $A$, and thus it can be identified with a section of $A_A$.
If, for every $s\in \Gamma(A)$ and $X\in\field{B}$, this is the pullback of a (unique) section of $A\to B$, which we suggestively denote $\nabla_X s$, then it can easily be checked that $(X,s)\mapsto \nabla_X s$ satisfies the defining properties of a linear connection.

Conversely, given a linear connection $\nabla$ on $A$, we have a pullback connection $\nabla_{p^A_B}$ on the tautological pullback bundle $A_A$ defined by its action on pullback sections $s_A$ associated to sections $s$ of $A$, namely
\begin{equation}
	\left(\nabla_{p^A_B}\right)_X s_A= (\nabla s)_{p^A_B}\left((p^A_B)_*X\right).
\end{equation}
For clarity, we shall omit $p^A_B$ in $\nabla_{p^A_B}$ and write the pullback connection associated to a connection $\nabla$ as simply $\nabla$.

With the help of the pullback connection on $A_A$ and the tautological section $\Phi_A$ of $A_A$, we obtain a section $\nabla\Phi_A \in \Gamma(T^*A \otimes A_A)$.
This can be interpreted as a map $TA \rightarrow A_A$ that is left-inverse to $\vt$ in \eqref{diag:exact} and so induces a splitting $TA \cong \mathcal V \oplus \mathcal H$.
This establishes the equivalence.

Since principal connections $\varphi$ and linear connections $\nabla$ are both special cases of Ehresmann connections, they provide a notion of horizontal lifts, which we shall denote $X^\varphi$ and $X^\nabla$ respectively.

Notice that if $\varphi, \varphi'$ are two principal connections, and $\nabla, \nabla'$ are two linear connections such that $\varphi - \varphi' = (p^A_B)^*\theta$ and $\nabla - \nabla' = C$, then the difference of the horizontal lifts is given by
\begin{align}
	X^\varphi - X^{\varphi'} = -\theta(X)^{\circ},
	&\hfill&
	X^\nabla - X^{\nabla'} = -\vt(C_X(\Phi_A)).
\end{align}

A linear connection on a complex vector bundle with a real structure is called \emph{real} when it maps real sections to real sections.
Finally, we will denote the Levi-Civita connection on a Riemannian manifold $(M,g)$ as $\nabla^{g}$.

\subsection{Complete lifts}

Suppose we are given a vector field $X$ on $B$. Let $\psi_t$ be its flow. Then, we may define a $1$-parameter subgroup of diffeomorphisms on the tangent bundle $TB$ by $(u,X_u)\mapsto (\psi_t(u), (\psi_t)_*X_u)$. This is the flow of a vector field $X^T$ on $TB$, which is a lift of $X$. The vector field $X^T$ is said to be the \emph{complete} lift of $X$ cf.\ \cite{YI}. We would like to relate it to the horizontal lift with respect to a connection.

To do this, we first make some general remarks about vector fields on the total space of a vector bundle $A\rightarrow B$. These can be regarded as derivations on the ring of smooth functions on $A$.
By Taylor's theorem, a derivation on $\smooth{A}$ is completely specified by its action on the ring of functions on $A$, which are polynomial along the fibres.
Such functions may be identified with sections of the symmetrised dual tensor bundle $\mathrm{Sym}_\bullet A^*\rightarrow B$, where
\begin{equation}
	\mathrm{Sym}_\bullet A^* \colonequals \bigoplus_{j=0}^\infty \mathrm{Sym}(A^{*\otimes j}).
\end{equation}
A section $\Xi=\sum_{j=0}^\infty \Xi_j$ of $\mathrm{Sym}_\bullet A^*$ corresponds to the function on the total space $A$ obtained by
\begin{equation}
\sum_{j=0}^\infty (\Xi_j)_A(\Phi_A,\dots,\Phi_A).
\end{equation}
A vector field on $A$ can thus be described in terms of its action as a derivation on sections of $\mathrm{Sym}_\bullet A^*$.

Given a section $s$ of $A_A$, the associated vertical vector field $\vt(s)$ corresponds to the derivation $\Xi_j \mapsto j\iota_s\Xi_j=j\Xi_j(s,\dots)$ for $\Xi_j\in\Gamma(\Sym(A^{*\otimes j}))$.
In particular, $\vt(\Phi_A)\in\field{A}$ corresponds to the derivation acting on $\mathrm{Sym}(A^{*\otimes j})$ as the scalar multiplication by $j$. Note that this is the derivation generated by the identity map on fibrewise linear functions $\alpha$, i.e.\ sections of $A^*$. More generally, the derivation generated by $\alpha \mapsto \alpha \circ \phi$ for $\phi \in \Gamma(\mathrm{End}(A)_A)$ may be identified with the vertical vector field $\vt(\phi(\Phi_A))$.

Let us now consider the horizontal lift $X^\nabla$ of a vector field $X$ on $B$ with respect to a linear connection $\nabla$. Its flow corresponds to the parallel transport along the flow lines of $X$. This in turn defines on sections of $\Sym_{\bullet}(A^*)$ a derivation $\Xi \mapsto \nabla_X \Xi$.

Finally, let us go back to the case where $A$ is the tangent bundle $TB \rightarrow B$.
Here, the complete lift $X^T$ of a vector field $X$ on $B$, acts by Lie derivative on the fibres, and this corresponds to the derivation $\Xi \mapsto \Lie{X} \Xi$.

Now, we see that when $\nabla$ is torsion-free, the difference $X^T - X^\nabla$ acts on a fibrewise linear function on $TB$, identified with a $1$-form $\alpha$ on $B$, as 
\begin{equation}
	\alpha \longmapsto \Lie{X} \alpha - \nabla_X \alpha = \alpha(\nabla X).
\end{equation}
As we have seen above, the derivation generated by this corresponds to the vector field $\vt((\nabla X)(\Phi_{TB}))$, where $\nabla X$ is interpreted as an endomorphism field. Thus, we have
\begin{equation}\label{eq:complete-split}
	X^T = X^\nabla +\vt((\nabla X)(\Phi_{TB})).
\end{equation}
This equation can also be generalised to connections with torsion $T^{\nabla}$ by writing $\nabla X+\iota_X T^{\nabla}$ instead of $\nabla X$.
A similar relation was obtained for cotangent bundles in \cite[Lemma 3.12 (p.\ 111)]{CST} by working in local coordinates.

\section{Special Kähler geometry}\label{sec:SK}
In this section, we explicitly describe how to translate the data of a projective special K\"ahler manifold to that of a variation of Hodge structure of a certain type (Proposition \ref{pr:VPHS1}) and vice versa (Proposition \ref{pr:VPHS2}).
In order to accomplish this, we will need to establish a few technical lemmata.
We do this in \S \ref{ssec:Formulas_and_properties}.
\subsection{Basic definitions}

We begin by reviewing the definitions of affine special Kähler, conic special Kähler, and projective special Kähler manifolds.

\begin{defi}\label{def:ASK}
	An \emph{affine special Kähler} (ASK) manifold is the data of a pseudo-Kähler manifold $(\widetilde{M},\wt g,I,\wt \omega)$ with a flat, torsion-free, symplectic connection $\wt\nabla$ such that $(\wt\nabla_X I)Y =(\wt\nabla_Y I)X$ for all vector fields $X,Y$ on $\wt M$.
\end{defi}
\begin{rmk}
	A consequence of torsion-freeness and $(\wt\nabla_X I)Y =(\wt\nabla_Y I)X$ is that $\wt\omega((\wt \nabla - \nabla^{\wt g})_X Y, Z)$ and
	\begin{equation}\label{eq:ask1}
	\begin{split}
		\wt \omega((\wt \nabla - \nabla^{\wt g})_X (IY), Z)&=\wt\omega((\wt \nabla_X I)Y, Z) + \wt\omega(I(\wt \nabla - \nabla^{\wt g})_X Y, Z)\\
		&=\wt\omega((\wt \nabla_X I)Y, Z) - \wt\omega((\wt \nabla - \nabla^{\wt g})_X Y, IZ)
	\end{split}
	\end{equation}
	are symmetric under the exchange $X \leftrightarrow Y$.
	Meanwhile, the symplectic condition $\wt \nabla \wt \omega = 0$ implies that $\wt\omega((\wt \nabla - \nabla^{\wt g})_X Y, Z)$ is symmetric also under the exchange $Y\leftrightarrow Z$, and hence fully symmetric.
	This, in turn, implies full symmetry for $\wt \omega((\wt \nabla - \nabla^{\wt g})_X (IY), Z)$ as well, since we have
	\begin{equation}
	\begin{split}
		\wt \omega((\wt \nabla - \nabla^{\wt g})_X (IY), Z)
		&= \wt \omega((\wt \nabla - \nabla^{\wt g})_{IY} X, Z)
		= \wt \omega((\wt \nabla - \nabla^{\wt g})_{IY}Z, X)\\ &= \wt \omega((\wt \nabla - \nabla^{\wt g})_Z (IY), X).
	\end{split}
	\end{equation}
	In particular, by replacing $X$ with $IX$ in \eqref{eq:ask1} and using the symmetries just proved, we deduce
	\begin{equation}
		(\wt \nabla - \nabla^{\wt g})_X Y = -\frac{1}{2}(\wt \nabla_{IX}I)Y= -\frac{\I}{2}(\wt \nabla^{1,0}_{X}I)Y +\frac{\I}{2}(\wt \nabla^{0,1}_{X}I)Y.
	\end{equation}
	We shall henceforth denote the $(1,0)$ part of the above as $\wt \eta_X Y$.
	This defines a section $\wt \eta$ of $T^*_{1,0} M\otimes T M\otimes T^*M$.
	Note that the symmetry property means $\wt\eta$ is in fact a section of $\sharp_2 S_{3,0}\wt M$, where $S_{3,0}\wt M$ denotes the subbundle of $(T^*M)^{\otimes 3}$ of symmetric tensors of type $(3,0)$, and $\sharp_2$ indicates that the second index is raised using the metric $\wt g$.
	
	The flatness of $\wt \nabla$ implies that $\wt \eta$ satisfies $\ol\partial^{\wt g} \wt\eta=0$, where $\ol\partial^{\wt g}$ is the $(0,1)$ part of the exterior covariant derivative with respect to $\nabla^{\wt g}$, i.e.\ 
\begin{equation}
	\ol\partial^{\wt g}\colon \Omega^{1,0}\left(\wt M,T^{0,1}\wt M\otimes T^*_{1,0}\wt M\right)\longrightarrow \Omega^{1,1}\left(\wt M,T^{0,1}\wt M\otimes T^*_{1,0}\wt M\right).
\end{equation}
Since $\nabla^{\wt g}$ respects the $(1,0)$, $(0,1)$ decomposition, this is a Dolbeault operator.
We shall thus refer to $\wt \eta$ as the \emph{holomorphic difference tensor field}.
For later applications, it will be useful to consider $\wt \eta^{\flat}\colonequals \wt g(\wt \eta,\cdot)$, which hence is a holomorphic symmetric tensor of type $(3,0)$.

It is possible to give a characterisation of ASK manifolds in terms of $\wt\eta$, \cite[Proposition 1.34 (p.\ 39)]{Freed1999}.
\end{rmk}
\begin{defi}\label{def:CASK}
	A \emph{conic special Kähler} (CSK) manifold is the data of a pseudo-Kähler manifold $(\widetilde{M},\wt g,I,\wt \omega)$ with a flat, symplectic connection $\wt\nabla$ and a vector field $\xi$ such that
	\begin{enumerate}
		\item $\xi$ is nowhere vanishing;
		\item $\wt g$ is negative definite on $\spn{\xi,I\xi}_{\R}$ and positive definite on its orthogonal complement;
		\item $\wt \nabla \xi= \nabla^{\wt g} \xi=\id$;
		\item $\wt \nabla (I\xi)=\nabla^{\wt g} (I\xi)=I$.
	\end{enumerate}
\end{defi}
We adopt the convention $\wt \omega=\wt g(I\cdot,\cdot)$.

\begin{rmk}
	Conical special Kähler manifolds are automatically affine special Kähler. Torsion-freeness follows from the computation
	\begin{equation}
		\wt\nabla_X Y - \wt \nabla_Y X - [X,Y]
		=\wt\nabla_X\wt \nabla_Y\xi -\wt \nabla_Y \wt\nabla_X \xi -\wt\nabla_{[X,Y]}\xi
		= R^{\wt\nabla}(X,Y)\xi = 0.
	\end{equation}
	Meanwhile, $(\wt\nabla_XI) Y = (\wt\nabla_Y I) X$ follows from the computation
	\begin{equation}
		\begin{split}
		(\wt\nabla_XI) Y - (\wt\nabla_Y I) X &=\wt\nabla_X (IY) -I\wt\nabla_XY - \wt\nabla_Y (IX) + I\wt\nabla_YX\\
		&=\wt\nabla_X (IY) - \wt\nabla_Y (IX) - I[X,Y]\\
		&=\wt\nabla_X\wt\nabla_Y(I\xi) - \wt\nabla_Y \wt\nabla_X (I\xi) -\wt\nabla_{[X,Y]}(I\xi)\\
		& = R^{\wt\nabla}(X,Y)I\xi = 0.
		\end{split}
	\end{equation}
	Thus, Definition \ref{def:CASK} implies Definition 3 in \cite{CiSC} if we take $-\wt g$ as the metric.
	The opposite implication follows from \cite[Lemma 2.3 (p.\ 2647)]{MMPSK2019}.
\end{rmk}

\begin{rmk}
	The conic special structure may also be described in terms of the vector fields
	\begin{align}
	\zeta = \frac{1}{2}(\xi - \I I\xi),
	&\hfill&
	\ol\zeta = \frac{1}{2}(\xi + \I I\xi),
	\end{align}
	in the complexified tangent bundle.
	In particular, the defining properties 3 and 4 of CSK manifolds become
	\begin{align}
		\wt \nabla_X \zeta=\nabla_X^{\wt g} \zeta= X^{1,0},
	&\hfill&
	 \wt \nabla_X \ol\zeta=\nabla_X^{\wt g}\ol \zeta= X^{0,1}.
	\end{align}
\end{rmk}

\begin{rmk}
	The holomorphic difference tensor field $\wt \eta$ of a CSK manifold is horizontal, since
	\begin{equation}
	\begin{split}
		\wt \eta_Y \zeta &= \wt \nabla^{1,0}_Y\zeta - (\nabla^{\wt g})_Y^{1,0}\zeta = Y^{1,0}-Y^{1,0}=0, \\
		\wt \eta_Y \ol\zeta &= \wt \nabla^{1,0}_Y\ol\zeta - (\nabla^{\wt g})_Y^{1,0}\ol\zeta = 0-0=0.
	\end{split}
	\end{equation}
\end{rmk}
In this paper, we will only consider CSK manifolds for which $\xi$ and $I\xi$ generate a $\C^\times$-action. This property allows us to take a symplectic reduction with respect to the $\U(1)$-action generated by $I\xi$.

\begin{defi}\label{def:PSK}
	A \emph{projective special Kähler} (PSK) manifold is a Kähler manifold $M$ endowed with a principal $\C^\times$-bundle $\pi\colon \widetilde{M}\to M$ with $(\widetilde{M},\wt g,I,\wt \omega,\wt\nabla,\xi)$ conic special Kähler such that $\xi$ and $I\xi$ are the fundamental vector fields associated to $1,\I\in\C$ respectively and $M$ is the Kähler quotient with respect to the induced $\U(1)$-action.
	In this case, we say that $M$ has a projective special Kähler structure.
	We call $\xi$ the \emph{Euler vector field}.
\end{defi}

\begin{defi}
	A \emph{PSK isomorphism} i.e.\ an isomorphism between PSK manifolds $(\pi\colon\widetilde{M}\rightarrow M,\wt g,I,\wt\omega,\wt\nabla,\xi)$ and $(\pi'\colon\widetilde{M'}\rightarrow M',\wt g',I',\wt\omega',\wt\nabla',\xi')$ is a pair $(\wt \psi, \psi)$ consisting of (bijective) Kähler isometries $\wt\psi\colon \wt M\rightarrow \wt M'$ and $\psi\colon M\rightarrow M'$ such that the following diagram commutes.
	\begin{equation}
	\begin{tikzcd}
	\wt M\arrow[r, "\wt \psi"]\arrow[d,"\pi"']& \wt M' \arrow[d, "\pi'"] \\
	M \arrow[r, "\psi" ]& M'
	\end{tikzcd}
	\end{equation}
	Moreover, $\wt \psi$ is compatible with the connections and the Euler vector fields.
	Stated in a more explicit manner,
\begin{align}
\wt\nabla'_{\wt \psi}\circ \wt \psi_*=\wt\psi_*\circ \wt\nabla,
&\hfill&
\wt \psi_* \xi =\xi'_{\wt\psi}\ .
\end{align}
A PSK isomorphism is called a \emph{PSK automorphism} when the PSK data of the two PSK manifolds coincide.
\end{defi}

	For brevity, we will often denote a PSK manifold as simply $\pi\colon \widetilde{M}\to M$ and suppress the remaining structure. 
	 
	\begin{rmk}\label{rmk:principal_C*_connection}
		The bundle $\pi\colon \widetilde{M}\to M$ has a principal $\C^{\times}$-connection $\chi$ induced by the metric $\wt g$ via
		\begin{equation}\label{eq:principal_C*_connection}
		\chi
		=\frac{\widetilde{g}(\xi,\cdot)+\I\widetilde{g}(I\xi,\cdot)}{\widetilde{g}(\xi,\xi)}.
		\end{equation}
		Therefore, in terms of the principal connection, the conditions $\nabla^{\wt g}\xi = \id$ and $\nabla^{\wt g}(I\xi) = I$ are equivalent to the condition that the metric $\wt g$ is of the form
		\begin{equation}
			\wt g = -\wt g(\xi,\xi)(\pi^*g_M-\ol \chi \chi)
			=-\wt h(\ol \zeta, \zeta) (\pi^*g_M-\ol \chi \chi),
		\end{equation}
			where $\wt h= \wt g + \I \wt \omega$.
			
		Note that $\wt \varphi\colonequals \mathrm{Im}(\chi)$ restricts to a principal connection on the principal $\U(1)$-bundle $ S\rightarrow M$ obtained by taking the level set $g(\xi,\xi)=-1$.
		
		If we further let $r^2=-\wt g(\xi,\xi)$, then $\wt g$ is a conical metric on the cone $S \times \R_{>0}$ over $S$, that is
		\begin{equation}
			\wt g = r^2 g_M - r^2\wt \varphi^2-\dif r^2.
		\end{equation}
	\end{rmk}
	\begin{rmk}\label{rmk:Mtilde_and_L}
		From the bundle $\pi\colon\wt M\to M$, we can build the complex line bundle associated to the standard $\C^\times$-representation on $\C$, i.e.\
		\begin{equation}
			L\colonequals \widetilde{M}\times_{\C^\times}\C\longrightarrow M.
		\end{equation}
		Since $\pi$ is holomorphic, $L$ inherits a holomorphic structure.
		Moreover, $L$ also has a hermitian structure $h_L$ associated to $r^2\langle\cdot,\cdot\rangle$, where $\langle\cdot,\cdot\rangle$ is the standard hermitian form on $\C$.

		Notice that, using the map $u\in\widetilde{M}\mapsto [u,1]\in L$, we can regard $\widetilde{M}$ as $L$ with the zero section removed.
		In this interpretation, given a local section $s$ of $L$, the function $h_L(\ol s,s)$ is the pullback of $r^2$ along $s$.
		
		Finally, the Chern connection for $L$ is precisely the linear connection $\nabla^\chi$ associated to the principal connection $\chi$, i.e.\ the one defined by 
\begin{equation}
		\nabla^\chi \Phi_L=\chi\otimes\Phi_L.		
\end{equation}
	\end{rmk}

\subsection{Formulas and properties}
\label{ssec:Formulas_and_properties}

In this section, we establish certain formulas and properties of PSK manifolds that we shall be making use of later.

\begin{lemma}\label{lemma:derivative_chi}
	The principal $\C^\times$-connection on a PSK manifold $\pi\colon\widetilde{M}\to M$ defined in Remark \ref{rmk:principal_C*_connection} satisfies
	\begin{equation}
	\widetilde{\nabla}\chi=\nabla^{\widetilde{g}}\chi=-\chi^2-\pi^* h_M,
	\end{equation}
	where $h_M= g_M + \I\omega_M$.
\end{lemma}
\begin{proof}
	We first write $\chi$ as 
	\begin{equation}
		\chi =\frac{1}{r^2}\,(\widetilde{\omega}(I\xi,\cdot)-\I\widetilde{\omega}(\xi,\cdot)).
	\end{equation}
	Taking its covariant derivative, we get
	\begin{equation}
	\begin{split}
	\widetilde{\nabla}\chi
	&=\widetilde{\nabla}\left( \frac{\widetilde{\omega}(I\xi,\cdot)-\I\widetilde{\omega}(\xi,\cdot)}{r^2}\right)
	=-2\,\frac{\dif r}{r}\,\otimes \chi+\frac{\widetilde{\omega}(\widetilde{\nabla}(I\xi),\cdot)-\I \widetilde{\omega}(\widetilde{\nabla}\xi,\cdot)}{r^2}\\
	&=-(\chi+\ol \chi)\otimes \chi+\frac{\widetilde{\omega}(I,\cdot)-\I \widetilde{\omega}}{r^2}
	=-(\chi+\overline{\chi})\otimes \chi-\frac{\widetilde{h}}{r^2}\\
	&=-\chi^2-\frac{\widetilde{h}+r^2\ol \chi \otimes \chi}{r^2}
	=-\chi^2-\pi^* h_M.
	\end{split}
	\end{equation}
 The formula for $\nabla^{\widetilde{g}}$ is deduced similarly.
\end{proof}

	Since $\wt \nabla$ is torsion-free, the alternating part of $2\wt \nabla \chi$ is $\dif \chi$.
	Thus, as a corollary of the above computation we have $\dif\chi=-2\I\pi^*\omega_M$.
	The factor of $-2\I$ may seem strange when this is compared to the curvature of a principal $\U(1)$-connection, but the $\I$ is due to the fact that the (infinitesimal) $\U(1)$-action is given by the imaginary part of the (infinitesimal) $\C^\times$-action, while the $-2$ is due to the choice of level set in the Kähler quotient by which $\omega_M$ is defined.

\begin{lemma}\label{lemma:chi_invariance}
	The principal $\C^\times$-connection on a PSK manifold $\pi\colon\widetilde{M}\to M$ satisfies 
	\begin{equation}
	\Lie{\zeta}\chi=\Lie{\ol \zeta}\chi=\Lie{\zeta}\ol\chi=\Lie{\ol \zeta}\ol\chi=0.
	\end{equation}
\end{lemma}
\begin{proof}
	The formulas for $\chi$ follow from the Cartan formula, i.e.\ 
	\begin{equation}
	\begin{split}
	\Lie{\zeta}\chi &= \dif(\iota_\zeta\chi) + \iota_\zeta\dif\chi = \dif(1) -2\I\iota_\zeta\pi^*\omega = 0,\\
	\Lie{\ol\zeta}\chi &= \dif(\iota_{\ol\zeta}\chi) + \iota_{\ol\zeta}\dif\chi = \dif(0) -2\I\iota_{\ol\zeta}\pi^*\omega = 0.
	\end{split}
	\end{equation}
	The formulas for $\ol\chi$ follow similarly.
\end{proof}
\begin{lemma}
	Given a PSK manifold $\pi\colon\wt M \rightarrow M$, the Lie derivatives of the connections $\nabla^{\wt g}$ and $\wt \nabla$ on $\wt M$ satisfy
	\begin{align}
		\Lie{\zeta}\nabla^{\wt g} &= \Lie{\ol\zeta}\nabla^{\wt g} =0,&
		\hfill&
		\Lie{\zeta}\wt\nabla =-\Lie{\ol\zeta}\wt\nabla =-\frac{\I}{ 2}\, \wt\nabla I.
	\end{align}
\end{lemma}
\begin{proof}
	First of all, we derive an expression for the Lie derivative of a torsion-free connection:
	\begin{equation}
		\begin{split}
		(\Lie{X}\nabla)_Y Z &= \Lie{X}(\nabla_Y Z)-\nabla_{\Lie{X}Y}Z - \nabla_Y(\Lie{X}Z)\\
		&=\nabla_X\nabla_YZ -\nabla_{\nabla_YZ}X - \nabla_{[X,Y]}Z - \nabla_Y\nabla_XZ + \nabla_Y\nabla_ZX\\
		&=R^\nabla(X,Y)Z + \nabla^2_{Y,Z}X.
		\end{split}
	\end{equation}
	In particular, the Lie derivative of a torsion-free connection is tensorial. 
	
	In the case of $\wt \nabla$, the curvature $R^{\wt \nabla}$ vanishes, so we have
	\begin{equation}
	\begin{split}
	\Lie{\zeta}\wt \nabla &= \wt \nabla(\wt \nabla \zeta) = \frac{1}{2}\,\wt \nabla(\id - \I I)=-\frac{\I}{2}\,\wt \nabla I,\\
	\Lie{\ol \zeta}\wt \nabla &= \wt \nabla(\wt \nabla \ol\zeta) = \frac{1}{2}\,\wt \nabla(\id + \I I)=\frac{\I}{2}\,\wt \nabla I.
	\end{split}
	\end{equation}
	In the case of $\wt \nabla$, we instead make use of the fact that $\zeta$ and $\ol\zeta$ are conformal Killing vector fields, as we can see by computing
	\begin{equation}
		\Lie{\zeta}\wt g = \Lie{\ol\zeta }\wt g = \wt g.
	\end{equation}
	We may now use the fact that the Levi-Civita connection preserves the metric to deduce that, for any vector field $X$ such that $\Lie{X}\wt g = c\wt g$ for some constant $c$, the action of $\Lie{X}\nabla^{\wt g}$ as a derivation on $\wt g$ is
	\begin{equation}
		(\Lie{X}\nabla^{\wt g})\wt g = \Lie{X}(\nabla^{\wt g}\wt g) -\nabla^{\wt g}(\Lie{X}\wt g)= \Lie{X}(0) - c\nabla^{\wt g}\wt g=0.
	\end{equation}
	More explicitly, this means that for any conformal Killing vector field $X$, the tensor field $\Lie{X}\nabla^{\wt g}$ satisfies
	\begin{equation}
	\wt g\left((\Lie{X}\nabla^{\wt g})_Y Z,W\right)+\wt g\left((\Lie{X}\nabla^{\wt g})_Y W,Z\right)=0,
	\end{equation}
	for arbitrary vector fields $Y$, $Z$, $W$.
	Meanwhile, by torsion-freeness, $(\Lie{X}\nabla^{\wt g})_Y Z$ is symmetric in $Y$ and $Z$.
	Combining these observations, we see that
	\begin{equation}
	\wt g\left((\Lie{X}\nabla^{\wt g})_Y Z,W\right)
	=-\wt g\left((\Lie{X}\nabla^{\wt g})_Y W,Z\right)
	=-\wt g\left((\Lie{X}\nabla^{\wt g})_W Y,Z\right).
	\end{equation}
	This is a circular permutation in $Y$, $Z$, and $W$, which changes the overall sign.
	Applying this permutation thrice, we obtain
	\begin{equation}
	\wt g\left((\Lie{X}\nabla^{\wt g})_Y Z,W\right)
	=-\wt g\left((\Lie{X}\nabla^{\wt g})_Y Z,W\right)
	=0.
	\end{equation}
	By non-degeneracy of $\wt g$, we conclude that $\Lie{X}\nabla^{\wt g}=0$.
\end{proof}

	As a corollary of the above result, we have that
	\begin{align}
		\Lie{\zeta}\wt \eta = -\frac{\I}{2}\,\wt \nabla^{1,0} I = \wt \eta,&
		\hfill&
		\Lie{\ol\zeta}\wt \eta = \frac{\I}{2}\,\wt \nabla^{1,0} I = -\wt \eta.
	\end{align}
	In particular, we can compute the Lie derivatives of the tensor field $\wt \eta^{\flat}$, obtaining
	\begin{equation}
	\begin{split}
		\Lie{\zeta}(\wt \eta^{\flat})
		&=\Lie{\zeta}(\wt g(\wt \eta,\cdot))
		=(\Lie{\zeta}\wt g)(\wt \eta,\cdot) + \wt g(\Lie{\zeta}\wt \eta,\cdot) 
		= 2\wt g(\wt \eta,\cdot)=2\wt \eta^{\flat},\\
		\Lie{\ol\zeta}(\wt \eta^{\flat})
		&=\Lie{\ol\zeta}(\wt g (\wt \eta,\cdot))
		= (\Lie{\ol\zeta}\wt g)(\wt \eta,\cdot) + \wt g (\Lie{\ol\zeta}\wt \eta,\cdot) = 0.
	\end{split}
	\end{equation}
	Thus $\wt \eta^{\flat}$ varies quadratically with respect to the $\C^\times$-action on $\wt M$. We have already seen that it is horizontal.
	 As a consequence, we get a (linear) bundle map $L \otimes L \rightarrow S_{3,0}M$. Raising the second index using the metric $g_M$ on $M$ then gives us a map
	\begin{equation}
		\eta \colon L \otimes L \longrightarrow \sharp_2S_{3,0}M.
	\end{equation}
	This map was constructed in \cite[Proposition 6.5 (p.\ 16)]{MMPSK2019}, where it was denoted by $\widehat{\gamma}$, and referred to as the \emph{intrinsic deviance}. Explicitly, we have
	\begin{equation}
	\wt \eta^{\flat}=\pi^*(g_M(\eta(\Phi_L,\Phi_L),\cdot)).
	\end{equation}

\begin{lemma}\label{lemma:LC_difference}
	Given a PSK manifold $\pi\colon\wt M \rightarrow M$ and vector fields $X,Y$ on $\wt M$, we have the equation
	\begin{equation}\label{eq:LC_difference}
	\pi_*(\nabla^{\widetilde{g}}_X Y)-\nabla^g_{X} \pi_* Y
	=\pi_*\big(\chi(X)Y^{1,0}+\ol \chi(X)Y^{0,1}+\chi(Y)X^{1,0}+\ol \chi(Y)X^{0,1}\big)
	\end{equation}
	holding on the pullback bundle $(TM)_{\pi}\to \wt M$.
\end{lemma}
\begin{proof}
Notice that both sides of \eqref{eq:LC_difference} are $\smooth{\wt M}$-linear in $X$ and $Y$, so it is enough to verify the formula pointwise.
Consider thus the tensor
\begin{equation}
2g\big(\pi_*(\nabla^{\widetilde{g}}_X Y)-\nabla^{g}_{X} \pi_*Y, Z'\big),
\end{equation}
for a vector field $Z'$ on $M$.
Let $Z=(Z')^{\chi}$ be the horizontal lift of $Z'$, then
\begin{equation}
	\begin{split}
	&2g\big(\pi_*(\nabla^{\widetilde{g}}_X Y)-\nabla^{g}_{X} \pi_*Y, Z'\big)
	=2\pi^*g\big(\nabla^{\widetilde{g}}_X Y),Z\big)
	-2g\big(\nabla^{g}_{X} \pi_*Y,Z'\big)\\
	&\qquad=2\,(r^{-2}\wt g+\ol \chi\chi)(\nabla^{\widetilde{g}}_X Y,Z)-2 g(\nabla^{g}_{X} \pi_*Y,Z')\\
	&\qquad=\frac{2}{r^2}\,\widetilde{g}(\nabla^{\widetilde{g}}_X Y,Z)-2 g(\nabla^{g}_{\pi_* X} \pi_* Y,Z').
	\end{split}
	\end{equation}
Being tensorial, $X$ and $Y$ can be assumed to be $\pi$-related to vector fields on $M$.
So let $X',Y'\in\field{M}$ be such that $\pi_* X=X'_\pi$ and $\pi_* Y=Y'_\pi$.
We now obtain
\begin{equation}
2g\big(\pi_*(\nabla^{\widetilde{g}}_X Y)-\nabla^{g}_{X} \pi_*Y, Z'\big)
=\frac{2}{r^2}\,\widetilde{g}(\nabla^{\widetilde{g}}_X Y,Z)-2 g(\nabla^{g}_{X'} Y',Z')_{\pi}.
\end{equation}
By the Koszul formula for $\widetilde{g}$ and $g$, this expression can be written as
\begin{equation}
\begin{split}
&\frac{2}{r^2}\big(X(\widetilde{g}(Y,Z))+Y(\widetilde{g}(Z,X))-Z(\widetilde{g}(X,Y))\\
&\quad+\widetilde{g}([Z,X],Y)+\widetilde{g}([X,Y],Z)-\widetilde{g}([Y,Z],X)\big)\\
&\quad- X'(g(Y',Z'))_{\pi}- Y'(g(Z',X'))_{\pi}+ Z(g(X',Y'))_{\pi}\\
&\quad- g([Z',X'],Y')_{\pi}-g([X',Y'],Z')_{\pi}+g([Y',Z'],X')_{\pi}.
	\end{split}
	\end{equation}
The last terms can be simplified thanks to the equalities
\begin{align}
\begin{split}
X'(g(Y',Z'))_{\pi}
&=\dif(g(Y',Z'))_{\pi}(X'_{\pi})
=\pi^*\dif(g(Y',Z'))(X)\\
&=\dif(g(Y',Z')_{\pi})(X)
=\dif(\pi^*g(Y,Z))(X)
=X(\pi^*g(Y,Z)),\\
g([X',Y'],Z')_{\pi}
&=g_{\pi}(\pi_*[X,Y],\pi_*Z)
=\pi^*g([X,Y],Z),
\end{split}
\end{align}
and similarly for the other terms.
Thus, we obtain
	\begin{equation}
	\begin{split}
	&2g\big(\pi_*(\nabla^{\widetilde{g}}_X Y)-\nabla^{g}_{X} \pi_*Y, Z'\big)\\
	&\quad
	=2\,\frac{X(\wt g(Y,Z))
	+Y(\wt g(Z,X))
	-Z(\wt g(X,Y))}{r^2}
	+\bigg(\frac{\wt{g}}{r^2}-\pi^*g\bigg)([Z,X],Y)\\
	&\quad\quad
	+\bigg(\frac{\wt{g}}{r^2}-\pi^*g\bigg)([X,Y],Z)
	-\bigg(\frac{\wt{g}}{r^2}-\pi^*g\bigg)([Y,Z],X)\\
	&\quad\quad- X(\pi^*g(Y,Z))
	- Y(\pi^*g(Z,X))
	+ Z(\pi^*g(X,Y)).
	\end{split}
	\end{equation}
Using the expression $\wt g=r^2(\pi^*g-\ol \chi\chi)$, and the $\chi$-horizontality of $Z$, we deduce
	\begin{equation}
	\begin{split}
	&2g\big(\pi_*(\nabla^{\widetilde{g}}_X Y)-\nabla^{g}_{X} \pi_*Y, Z'\big)\\
	&\quad=2\,\frac{\dif r}{r}(X)(\pi^*g-\ol \chi\chi)(Y,Z)
	+2\,\frac{\dif r}{r}(Y)(\pi^*g-\ol \chi\chi)(Z,X)\\
	&\quad\quad-2\,\frac{\dif r}{r}(Z)(\pi^*g-\ol \chi\chi)(X,Y)
	+X((\pi^*g-\ol \chi\chi)(Y,Z))\\
	&\quad\quad+ Y((\pi^*g-\ol \chi\chi)(Z,X))
	- Z((\pi^*g-\ol \chi\chi)(X,Y))\\
	&\quad\quad-\ol \chi\chi([Z,X],Y)
	-\ol \chi\chi([X,Y],Z)
	+\overline{\chi}\chi([Y,Z],X)\\
	&\quad\quad- X(\pi^*g(Y,Z))
	- Y(\pi^*g(Z,X))
	+ Z(\pi^*g(X,Y))\\
	&\quad=2\,\frac{\dif r}{r}(X)\pi^*g(Y,Z)
	+2\,\frac{\dif r}{r}(Y)\pi^*g(Z,X)
	-2\,\frac{\dif r}{r}(Z)(\pi^*g-\ol \chi\chi)(X,Y)\\
	&\quad\quad
	+Z(\ol \chi\chi(X,Y))
	-\ol \chi\chi([Z,X],Y)
	+\ol \chi\chi([Y,Z],X).
	\end{split}
	\end{equation}
	We know that $\dif r/r$ is the real part of $\chi$, thus
	\begin{equation}\label{eq:Lie_Z}
	\begin{split}
	&2g\big(\pi_*(\nabla^{\widetilde{g}}_X Y)-\nabla^{g}_{X} \pi_*Y, Z'\big)\\
	&\qquad=(\chi+\ol \chi)(X)\pi^*g(Y,Z)
	+(\chi+\ol \chi)(Y)\pi^*g(Z,X)+0\\
	&\quad\quad
	+Z(\ol \chi\chi(X,Y))
	-\ol \chi\chi([Z,X],Y)
	+\ol \chi\chi([Y,Z],X)\\
	&\qquad=(\chi+\ol \chi)(X)\pi^*g(Y,Z)
	+(\chi+\ol \chi)(Y)\pi^*g(Z,X)
	+\Lie{Z}(\ol \chi\chi)(X,Y).
	\end{split}
\end{equation}
	We now need the Lie derivatives of $\chi$, $\ol \chi$, and $\ol \chi \chi$.
	By the Cartan formula,
	\begin{equation}
	\Lie{Z}\chi
	=\iota_Z d\chi+d\iota_Z\chi
	=-2\I\iota_Z\pi^*\omega, \qquad
	\Lie{Z}\ol \chi
	=\iota_Z d\ol \chi+d\iota_Z\ol \chi
	=2\I\iota_Z\pi^*\omega.
	\end{equation}
	The Lie derivative for $\ol \chi \chi$ then follows from a short computation:
	\begin{equation}
	\begin{split}
	\Lie{Z}(\ol \chi\chi)(X,Y)
	&=(\Lie{Z}\ol \chi\chi+\ol \chi\Lie{Z}\chi)(X,Y)\\
	&=(2\I(\iota_Z\pi^*\omega)\chi-2\I\ol \chi(\iota_Z\pi^*\omega))(X,Y)\\
	&=\I\pi^*\omega(Z,X)\chi(Y)
	+\I\pi^*\omega(Z,Y)\chi(X)\\
	&\quad-\I\ol \chi(X)\pi^*\omega(Z,Y)
	-\I\ol \chi(Y)\pi^*\omega(Z,X))\\
	& =-\I(\chi-\ol \chi)(X)\pi^*g(IY,Z)-\I(\chi-\ol \chi)(Y)\pi^*g(IX,Z).
	\end{split}
	\end{equation}
	Substituting the Lie derivatives we just computed into \eqref{eq:Lie_Z}, we get
	\begin{equation}
	\begin{split}
	&2g\big(\pi_*(\nabla^{\widetilde{g}}_X Y)-\nabla^{g}_{X} \pi_*Y, Z'\big)\\
	&\qquad=(\chi+\ol \chi)(X)\pi^*g(Y,Z)
	+(\chi+\ol \chi)(Y)\pi^*g(X,Z)\\
	&\qquad\quad-\I(\chi-\ol \chi)(X)\pi^*g(IY,Z)-\I(\chi-\ol \chi)(Y)\pi^*g(IX,Z)\\
	&\qquad=\chi(X)\pi^*g(Y-\I IY,Z)+\ol\chi(X)\pi^*g(Y+\I IY,Z)\\
	&\qquad\quad +\chi(Y)\pi^* g(X-\I IX,Z)+\ol \chi(Y)\pi^*g(X+\I IX,Z)\\
	&\quad\quad=2\pi^*g\big(\chi(X)Y^{1,0}+\ol \chi(X)Y^{0,1}+\chi(Y)X^{1,0}+\ol \chi(Y)X^{0,1},Z\big).
	\end{split}
	\end{equation}
	Hence,
	\begin{equation}
	\begin{split}
	&g\big(\pi_*(\nabla^{\widetilde{g}}_Y X)-\nabla^{g}_{\pi_* Y} \pi_*X, Z'\big)\\
	&\qquad=g\big(\pi_*\big(\chi(X)Y^{1,0}+\ol \chi(X)Y^{0,1}+\chi(Y)X^{1,0}+\ol \chi(Y)X^{0,1}\big),Z').
	\end{split}
	\end{equation}
	This proves the lemma, since $g$ is non-degenerate and $Z'$ arbitrary.
\end{proof}

\begin{lemma}\label{lemma:horizontal_action}
	Let $\pi\colon\widetilde{M}\to M$ be a PSK manifold with special connection $\wt \nabla$.
	Then, the $\wt \nabla$-horizontal lifts of the fundamental vector fields $\xi$ and $I \xi$ to $T\wt M$ generate a $\C^\times$-action on $T\wt M$.
	Moreover, a vector field $X$ on $\wt M$ is invariant under this action if and only if 
	\begin{equation}
	\wt \nabla_\xi X = \wt \nabla_{I\xi} X = 0.
	\end{equation} 	
\end{lemma}
\begin{proof}
	Let us denote the principal $\C^\times$-action on $\wt M$ as
	\begin{align}
		\rho_\lambda(u)\colonequals u \cdot \lambda,
		&\hfill&
		\lambda \in \C^\times.
	\end{align}
	Now consider the following principal $\C^\times$-action on $T\wt M$:
	\begin{equation}\label{eq:horizontal_lift_principal_action}
		X_u \longmapsto \lambda^{-1}(\rho_\lambda)_*X_u\in (TM)_{\rho_{\lambda}(u)}.
	\end{equation}
	This action is generated by vector fields $Y^1$ and $Y^\I$ given by
	\begin{equation}
		\begin{split}
		Y^1_{X_u}
		&=\left.\frac{\dif}{\dif t}( e^{-t}(\rho_{e^t})_*X_u)\right|_{t=0}
		=\left.\left((\rho_{e^t})_*\frac{\dif}{\dif t}(e^{-t}X_u)+ e^{-t}\frac{\dif}{\dif t}((\rho_{e^t})_*X_u)\right)\right|_{t=0}\\
		&=-\vt(\Phi_{T\wt M})_{X_u} + \xi^T_{X_u},\\
		Y^\I_{X_u}
		&=\left.\frac{\dif}{\dif t}(e^{-\I t}(\rho_{e^{\I t}})_*X_u)\right|_{t=0}
		=\left.\left((\rho_{e^{\I t}})_*\frac{\dif}{\dif t}( e^{-\I t} X_u)+ e^{-\I t}\frac{\dif}{\dif t}((\rho_{e^{\I t}})_* X_u)\right)\right|_{t=0}\\
		&=-\vt(I\Phi_{T\wt M})_{X_u} + (I\xi)^T_{X_u}.
		\end{split}
	\end{equation}
	In the previous equation, we implicitly exploited the fact that \eqref{eq:horizontal_lift_principal_action} is the composition of two commuting actions: a scalar multiplication and a pushforward.
	Meanwhile, by the splitting of the complete lift in \eqref{eq:complete-split}, we have
	\begin{equation}
	\begin{split}
		\xi^T
		&= \xi^{\wt \nabla}
		+ \vt\big((\wt \nabla \xi)\Phi_{T\wt M}\big)
		= \xi^{\wt \nabla}
		+ \vt\big(\Phi_{T\wt M}\big),\\
		(I\xi)^T
		&= (I\xi)^{\wt \nabla}
		+ \vt\big((\wt \nabla (I\xi))\Phi_{T\wt M}\big)
		= (I\xi)^{\wt \nabla}
		+ \vt\big(I\Phi_{T\wt M}\big).
	\end{split}
	\end{equation}
	Thus, we find that $Y^1 = \xi^{\wt \nabla}$ and $Y^\I = (I\xi)^{\wt \nabla}$, so the horizontal lifts indeed generate a $\C^\times$-action on $T\wt M$. 

	Moreover, from our earlier discussion on connections in \S \ref{ssec:Connections}, we know that the flows of the horizontal lifts of $\xi$ and $I\xi$ are precisely the parallel transport along the integral curves of $\xi$ and $I\xi$. A section $X$ of $TM$ is invariant under such parallel transport if and only if $\widetilde \nabla_\xi X$ and $\widetilde \nabla_{I\xi} X$ vanish.
\end{proof}

\subsection{Variations of Hodge structure}

We will now build up the definition of variation of polarised Hodge structure in several steps. Although the definitions are motivated by the cohomology groups of a non-singular complex variety and their deformation theory, we take an abstract differential-geometric approach along the lines of \cite{Simpson}. In particular, we relax the integrality condition that is often included as part of the definition of Hodge structures. 

\begin{defi}[Hodge structure]
	A Hodge structure on a real vector space $V$ of weight $d\in \mathbb Z$ is a decomposition of its complexification 
	\begin{equation}
	V_\mathbb C = \bigoplus_{p+q=d} V^{p,q}
	\end{equation}
	satisfying $\ol{V^{p,q}} = V^{q,p}$. 
\end{defi}
We will only be considering Hodge structures of positive weight such that $V^{p,q} = 0$ whenever either $p$ or $q$ is negative. 

\begin{eg}\label{eg:running-eg}
	Let $K$ be a compact Kähler manifold. Then the Dolbeault theorem tells us that the $d$-th de Rham cohomology $H^d_{\mathrm{dR}}(K)$ admits a Hodge structure of weight $d$, namely
	\begin{equation}
		H^d_{\mathrm{dR}}(K;\C)
		=\bigoplus_{p+q=d} H^q(K;\Omega^p)
		= \bigoplus_{p+q=d}H^{p,q}_{\ol\partial}(K),
	\end{equation}
	where $\Omega^p$ denotes the sheaf of holomorphic $p$-forms and $\ol \partial$ denotes the Dolbeault operator.
\end{eg}

\begin{defi}[Polarised Hodge structure]
	A polarised Hodge structure of weight $d$ on a real vector space $V$, denoted by $(V_\C=\bigoplus V^{p,q},Q)$, is the data of a Hodge structure $V_\C=\bigoplus V^{p,q}$ of weight $d$ on $V$ together with a $\C$-bilinear form $Q$ on $V_\C$ satisfying
	\begin{align}\label{eq:polarisation}
	Q(v_1 ,v_2)&=(-1)^{d}Q(v_2 ,v_1) &\hfill&\text{ for }v_1,v_2 \in V_{\C};\nonumber\\
	Q(v_1 ,v_2 )&=0 &\hfill&\text{ for }v_1 \in V^{p,q},v_2 \in V^{p',q'},p\neq q';\\
	\I^{p-q}Q\left(\ol{v },v\right)&>0 &\hfill&\text{ for }v \in V^{p,q},\ v \neq 0.\nonumber
	\end{align}
\end{defi}

\begin{eg}
	Let the compact Kähler manifold $K$ considered in Example \ref{eg:running-eg} be of complex dimension $n$.
	Then the cohomology group $H^d_{\mathrm{dR}}(K)$, with $d\le n$, has a natural bilinear pairing given by
	\begin{equation}
		Q(\alpha,\beta) = \int_K \frac{1}{(n-d)!}\omega^{n-d}\wedge\beta \wedge \alpha.
	\end{equation}
	The Lefschetz operator is compatible with this pairing, and therefore it induces a bilinear pairing on the subspace of primitive cohomology.
	Notice that the primitive cohomology inherits a Hodge decomposition, and the induced bilinear pairing $Q$ on it satisfies all three conditions in \eqref{eq:polarisation} (cf.\ \S 7.1.2 of \cite{Voisin}), the last one being a consequence of the Hodge--Riemann bilinear relations.
	
	In particular, when $K$ is a simply connected Calabi--Yau manifold of complex dimension $3$, all $3$-forms are primitive, so the whole of $H^3_{\mathrm{dR}}(K)$ admits a polarised Hodge structure with the aforementioned pairing.
\end{eg}

Next we will consider families of such polarised Hodge structures that are parametrised by some complex manifold.

\begin{defi}[Variation of polarised Hodge structure]
	A variation of polarised Hodge structure (VPHS) $(E_{\C}=\bigoplus E^{p,q}, Q, \nabla)$ of weight $d$ on a complex manifold $M$ is a real vector bundle $E \rightarrow M$ together with a flat connection $\nabla$, a $\nabla$-parallel bilinear form $Q\in \Gamma(E^*_\C\otimes E^*_\C)$, and a decomposition of complex vector bundles
	\begin{equation}
	E_{\C} = \bigoplus_{p+q=d} E^{p,q},
	\end{equation}
	such that the fibre $E_m$ over a point $m\in M$ carries a polarised Hodge structure $((E_m)_{\C}=\bigoplus E^{p,q}_m,Q_m)$ of weight $d$ and $\nabla$ gives a map
	\begin{equation}\label{eq:grif-trans}
	\nabla\colon \Gamma(E^{p,q}) \longrightarrow \Omega^{0,1}(E^{p+1,q-1}) \oplus \Omega^{1,0}(E^{p,q}) \oplus \Omega^{0,1}(E^{p,q}) \oplus \Omega^{1,0}(E^{p-1,q+1}).
	\end{equation}
\end{defi}
The connection $\nabla$ is called \emph{Gau\ss--Manin connection}, while \eqref{eq:grif-trans} is called \emph{Griffiths transversality condition}.

\begin{eg}
Consider a holomorphic fibre bundle $K\to M$ such that fibres are (simply connected, compact) Calabi--Yau $3$-folds.
This can be regarded as a family $\{K_m\}_{m\in M}$ of Calabi--Yau $3$-folds holomorphically parametrised by points $m$ of $M$.
The de Rham cohomology groups $H_{\mathrm{dR}}^3(K_m)$, then produce a bundle of polarised Hodge structures of weight $3$.
	
	Moreover, for any contractible $U\subseteq M$, we can choose a family of submanifolds $C^i\subseteq K|_{U}$ such that $C^i_m\colonequals C^i\cap K_m$ is a compact submanifold of $K_m$ of real dimension $3$, and furthermore, $\{[C^i_m]\}_i$ is a basis of homology classes of the free part of $H_3(K_m)$.
	This gives a dual basis of $3$-forms $\{\alpha_i\}_i$ in $H^3_{\mathrm{dR}}(K_m)$.

	Two different choices of basis $\{C^i_m\subset K_m\}_i$ and $\{C'^i_m\subset K_m\}_i$ are related by an integral (and hence locally constant) linear transformation.
	Thus, the associated dual bases of $3$-forms $\{\alpha_i\}_i$ and $\{\alpha'_i\}_i$ are also related by a locally constant linear transformation.
	This then defines a flat connection on the bundle $\{H^3_{\mathrm{dR}}(K_m)\}_{m\in M}$ that preserves the natural polarisation $Q$.
	Deducing the Griffiths transversality condition requires a bit more work, cf.\ §10.2.2 of \cite{Voisin}.
\end{eg}

On a variation of Hodge structure with odd weight, there are two natural complex structures $I_{\mathrm G}$ and $I_{\mathrm W}$, respectively called the \emph{Griffiths} and \emph{Weil} complex structures, defined on a section $\sigma\in\Gamma(E^{p,q})$ by
\begin{align}
I_{\mathrm G} X=\sign(p-q)\I X,
&\hfill&
I_{\mathrm W} X=\I^{p-q} X.
\end{align}
These, in turn, give rise to two Hermitian structures $h_{\mathrm G}$ and $h_{\mathrm W}$ respectively called the \emph{Griffiths} and \emph{Weil} hermitian forms defined by
\begin{align}
h_{\mathrm G}=Q(\cdot,I_{\mathrm G}\cdot)+\I Q,
&\hfill&
h_{\mathrm W}=Q(\cdot,I_{\mathrm W} \cdot)+\I Q.
\end{align}
Notice that both forms have $Q$ as imaginary part.

We will now show that the notion of PSK structure has an equivalent reformulation in terms of certain VPHS of weight $3$. We mainly follow the chain of reasoning given by Freed in \cite[Proposition 4.6 (p.\ 46)]{Freed1999}, but work out the correspondence much more explicitly.

We first show that PSK structures $\pi\colon \wt M \rightarrow M$ give rise to a certain VPHS of weight $3$ on $M$.

\begin{prop}\label{pr:VPHS1}
	Let $\pi\colon \wt M \rightarrow M$ be a PSK manifold and let $L$ be the holomorphic line bundle associated to the principal $\C^\times$-bundle $\widetilde M \rightarrow M$ described in Remark \ref{rmk:Mtilde_and_L}.
	Then $M$ carries a VPHS of weight $3$ given by the direct sum
	\begin{equation}\label{eq:HD}
		L\oplus (L \otimes T^{1,0}M) \oplus (\ol{L} \otimes T^{0,1}M) \oplus \ol{L}.
	\end{equation}
	The Gau\ss--Manin connection is given by
	\begin{equation}\label{eq:GM}
	\begin{split}
	&\nabla_Y(\sigma_1 + \sigma_2 \otimes Z + \ol \varsigma_2 \otimes \ol W + \ol \varsigma_1)\\
	&\qquad= \nabla_Y^{\chi}\sigma_1 +h_M(Y^{0,1},Z)\sigma_2\\
	&\qquad\quad + \nabla_Y^{\chi,g_M}(\sigma_2 \otimes Z) + \sigma_1 \otimes Y^{1,0}+ (\ol h_L{}^{-1} \circ \ol \eta)_{Y^{0,1}}(\ol \varsigma_2 \otimes \ol W)\\
	&\qquad\quad + \nabla_Y^{\ol \chi,g_M}(\ol\varsigma_2 \otimes \ol W) + \ol\varsigma_1 \otimes Y^{0,1} + (h_L^{-1} \circ \eta)_{Y^{1,0}}(\sigma_2 \otimes Z)\\
	&\qquad\quad + \nabla_Y^{\ol \chi}\ol \varsigma_1 + h_M(W,Y^{1,0})\ol\varsigma_2,
	\end{split}
	\end{equation}
	where $\sigma_1,\sigma_2$ are sections of $L$; $Z$ is a section of $T^{1,0}M$; $\ol W$ is a section of $T^{0,1}M$; $\ol \varsigma_1,\ol \varsigma_2$ are sections of $\ol L$; and $Y$ is a vector field on $M$.
\end{prop}
\begin{proof}
	A PSK manifold $M$ may be given in terms of a principal $\C^\times$-bundle $\pi\colon \widetilde M \rightarrow M$ which is CSK and so equipped with a flat torsion-free connection $\widetilde \nabla$.
	Let $E$ be the quotient of $T\widetilde M$ by the $\widetilde \nabla$-horizontal lift of the principal $\C^\times$-action on $\widetilde M$ (cf.\ Lemma \ref{lemma:horizontal_action}).
	The complexification $E_\C$ is then the corresponding quotient of $T\widetilde M_\C$ and forms a complex vector bundle over $M$.
	We shall show that this is equipped with the Hodge decomposition as in \eqref{eq:HD}.
	
	In order to do so, we first need to describe the bundle $E_{\C}$ more explicitly.
	First of all, we prove that there is a bundle isomorphism
	\begin{equation}\label{eq:TMisoEpi}
	\phi\colon T\wt M_{\C} \longrightarrow \left(L\oplus (L \otimes T^{1,0}M) \oplus (\ol{L} \otimes T^{0,1}M) \oplus \ol{L}\right)_{\pi}.
	\end{equation}
	Recall that we can identify $\widetilde{M}$ with $L\smallsetminus \{0\}$, cf.\ Remark \ref{rmk:Mtilde_and_L}.
	With the above identification, it makes sense to talk of the tautological section $\Phi_L$ of the pullback bundle $L_{\pi}$.	
	We define $\phi$ separately on the holomorphic and antiholomorphic part by the maps
	\begin{align}
	\begin{split}
	\phi^{1,0}=\chi\otimes \Phi_L+\Phi_L\otimes \pi_*\colon T^{1,0}\widetilde{M} &\longrightarrow (L\oplus (L \otimes T^{1,0}M))_{\pi},\\
	\phi^{0,1}=\ol \Phi_L\otimes \pi_*+\ol\chi\otimes \ol\Phi_L\colon T^{0,1}\widetilde{M} &\longrightarrow ((\ol{L} \otimes T^{0,1}M) \oplus \ol{L})_{\pi},
	\end{split}
	\end{align}
	so that $\phi=\phi^{1,0}+\phi^{0,1}$.
	Explicitly,
	\begin{equation}
	\phi(X)
	= \chi(X)\Phi_L + \Phi_L \otimes \pi_*X^{1,0} + \ol\Phi_L \otimes \pi_*X^{0,1} + \ol \chi(X)\ol \Phi_L.
	\end{equation}
	We will now describe an inverse for $\phi^{1,0}$.
	The tautological section $\Phi_L$ is a global one for the bundle $L_{\pi}$, so, in order to describe a map from $(L\oplus L\otimes TM^{1,0})_{\pi}$ to $T^{1,0}\widetilde{M}$, it is enough to see how it acts on $f\Phi_L+\Phi_L\otimes X$ with $f\in\smooth{\widetilde{M},\C}$ and $X$ a section of $(T^{1,0}M)_{\pi}$.
	The inverse of $\phi^{1,0}$ is given by 
	\begin{equation}
	f\Phi_L+\Phi_L\otimes X\longmapsto f\zeta+X^{\chi}.
	\end{equation}
	Analogously, we may show that $\phi^{0,1}$ has the inverse
	\begin{equation}
	f\ol\Phi_{ L}+\ol\Phi_{ L}\otimes X\longmapsto f\ol\zeta+X^{\ol \chi}.
	\end{equation}
	Thus, $\phi$ is invertible.

	We will now prove that $\phi$ descends through the quotient to an isomorphism 
	\begin{equation}
	E_{\C}\cong L\oplus (L \otimes T^{1,0}M) \oplus (\ol{L} \otimes T^{0,1}M) \oplus \ol{L}.
	\end{equation}
	Let us recall that two vector bundles are isomorphic if and only if their sheaves of sections are.
	By Lemma \ref{lemma:horizontal_action}, we know that sections of $E$ are in bijective correspondence with sections $X\in\Gamma(T\widetilde M)$ such that $\widetilde \nabla X$ vanishes on the vertical vector fields $\xi$ and $I\xi$. Equivalently,
	\begin{equation}
	\widetilde \nabla_\zeta X = \widetilde \nabla_{\ol \zeta} X = 0.
	\end{equation} 
	This, in turn, is equivalent to saying that we have the Lie derivatives
	\begin{align}\label{eq:Lie_derivatives_sections}
	\Lie{\zeta} X
	= \widetilde \nabla_\zeta X - \widetilde \nabla_{X}\zeta
	= - X^{1,0},
	&\hfill&
	\Lie{\ol \zeta} X
	= \widetilde \nabla_{\ol \zeta} X - \widetilde \nabla_{X}\ol\zeta
	= - X^{0,1}.
	\end{align}
	Applying $\pi_*$ to these equations, we obtain for all $p\in M$, a system of differential equations for $\pi_* X$ interpreted as a function $\wt M_p\to T_p M$.
	By solving them, we can say that, given $u\in \wt M_p$ and $\lambda\in\C^\times$, we have
	\begin{equation}\label{eq:lambda-trans-1}
	(\pi_*X)_{\lambda u} = \lambda^{-1}(\pi_*X^{1,0})_u + \ol \lambda{}^{-1}(\pi_*X^{0,1})_u.
	\end{equation}
	The identification of $\wt M$ as a subbundle of $L$ also gives
	\begin{equation}\label{eq:invariance_tautological_form}
	(\Phi_L)_{\lambda u}
	=\lambda u
	=\lambda(\Phi_L)_u.
	\end{equation}
	If \eqref{eq:Lie_derivatives_sections} holds, then, by Lemma \ref{lemma:chi_invariance}, we have $\Lie{\zeta}(\chi(X))=-\chi(X^{1,0})=-\chi(X)$ and $\Lie{\ol\zeta}(\chi(X))=-\chi(X^{0,1})=0$, which imply
	\begin{equation}\label{eq:variation_chi(X)}
	\chi(X)_{\lambda u}=\lambda^{-1}\chi(X)_u.
	\end{equation}
	Similarly, with $\ol \chi$ we obtain
	\begin{equation}\label{eq:variation_chibar(X)}
	\ol \chi(X)_{\lambda u}=\ol \lambda{}^{-1}\ol \chi(X)_u.
	\end{equation}
	Combining \eqref{eq:variation_chi(X)}, \eqref{eq:variation_chibar(X)}, \eqref{eq:invariance_tautological_form} and \eqref{eq:lambda-trans-1}, we get
	\begin{equation}
	\phi(X)_{\lambda u} = \phi(X)_u.
	\end{equation}
	The same can be said for $\phi^{1,0}(X)$ and $\phi^{0,1}(X)$.
	In other words, $\phi(X)$ is the pullback of some section of
	\begin{equation}\label{eq:E_explicit}
	L\oplus (L \otimes T^{1,0}M) \oplus (\ol{L} \otimes T^{0,1}M) \oplus \ol{L}.
	\end{equation}
	Thus, we may identify this bundle with $E_{\C}$.
	In particular, we get a Hodge decomposition for $E_{\C}$ given by
	\begin{equation}
	E^{3,0} = L, \quad E^{2,1} = L \otimes T^{1,0}M, \quad E^{1,2} = \ol L \otimes T^{0,1}M, \quad E^{0,3} = \ol L.
	\end{equation}
	
	Next, we describe in terms of $\wt \omega$ the bilinear form $Q$, which for a VPHS of degree $3$ is skew-symmetric.
	Note that the symplectic form $\widetilde \omega$ is $\widetilde \nabla$-parallel, so it descends to a skew-symmetric bilinear form on $E_{\C}$.
	We define $Q$ to be this skew-symmetric bilinear form.
	
	Let $h_L$ and $h_M$ be the Hermitian forms defined on $L$ and $TM_\C$ respectively.
	Then, $-h_L + h_L \otimes h_M$ may be regarded as a bilinear form on $E_{\C}$.
	Notice that
	\begin{equation}
	\begin{split}
	&(-h_L + h_L \otimes h_M)(\phi(X),\phi(Y))\\
	&\qquad\qquad= -\ol \chi(X)\chi (Y)h_L(\ol\Phi_L, \Phi_L) +h_L(\ol \Phi_L, \Phi_L)\pi^*h_M(X^{0,1},Y^{1,0})\\
	&\qquad\qquad=-r^2\ol \chi(X)\chi(Y) + r^2\pi^*h_M(X,Y)
	= \widetilde h(X,Y).
	\end{split}
	\end{equation}
	Thus, $\Alt(-h_L + h_L \otimes h_M)$ corresponds to $\I\wt\omega$, from which we infer
	\begin{equation}
	Q = \I\, \Alt(h_L - h_L \otimes h_M).
	\end{equation}
	Since $h_M$ and $h_L$ are positive-definite, it follows that
	\begin{align}
	\I^{3-0}Q(\ol s_1, s_1)= \frac{1}{2}\,h_L(\ol s_1, s_1) > 0,
	&\hfill&
	\I^{2-1}Q(\ol s_2, s_2) =\frac{1}{2}\, (h_L \otimes h_M)(\ol s_2,s_2)
	\end{align}
	for all local non-vanishing sections $s_1$ and $s_2$ of $L$ and $L \otimes T^{1,0}M$ respectively.
	
	Finally, we describe the Gau\ss--Manin connection and show that it satisfies the Griffiths transversality condition. To describe the Gau\ss--Manin connection $\nabla$, we first show that any connection $\nabla'$ satisfying
	\begin{align}\label{eq:pullback-conn}
	\Lie{\zeta}\nabla' =-\frac{\I}{ 2}\, \nabla'I,
	& \hfill &
	\Lie{\ol\zeta}\nabla' =\frac{\I}{ 2}\, \nabla'I,
	\end{align}
	is the pullback of some connection on $E$.
	
	Let $X$ be a vector field on $\wt M$ descending to a section of $E$, and hence satisfying $\Lie{\zeta}X=-X^{1,0}$ and $\Lie{\ol\zeta}X=-X^{0,1}$.
	Let $Y$ be a vector field on $\wt M$ such that $\pi_*Y$ is the pullback of some section of $TM$, and thus $\Lie{\zeta}Y=\Lie{\ol\zeta}Y=0$.
	As a consequence, we get
	\begin{align}
	\begin{split}
	\Lie{\zeta}(\nabla'_YX)
	&= (\Lie{\zeta}\nabla')_YX + \nabla'_{\Lie{\zeta}Y}X + \nabla'_Y(\Lie{\zeta}X)
	= -\frac{\I}{2}(\nabla'_YI)X - \nabla'_YX^{1,0}\\
	&=-\frac{1}{2}[\id - \I I, \nabla'_Y]X - \frac{1}{2}\nabla'_Y((\id -\I I)X)\\
	&=-\frac{1}{2}(\id -\I I)(\nabla'_Y(X)) =- (\nabla'_Y(X))^{1,0},\\
	\Lie{\ol \zeta}(\nabla'_YX)
	&= (\Lie{\ol \zeta}\nabla')_YX + \nabla'_{\Lie{\ol\zeta}Y}X + \nabla'_Y(\Lie{\ol\zeta}X)
	= \frac{\I}{2}(\nabla'_YI)X - \nabla'_YX^{0,1}\\
	&=-\frac{1}{2}[\id + \I I, \nabla'_Y]X - \frac{1}{2}\nabla'_Y((\id +\I I)X)\\
	&=-\frac{1}{2}(\id +\I I)(\nabla'_Y(X)) =- (\nabla'_Y(X))^{0,1}.
	\end{split}
	\end{align}
	So, the vector field $\nabla'_YX$ also descends to a section of $E$.
	Thus, $\nabla'$ is a pullback connection.
	
	As both the Levi-Civita connection $\nabla^{\widetilde {g}}$ as well as the special connection $\widetilde{\nabla}$ satisfy the property \eqref{eq:pullback-conn}, they are pullbacks of connections $\widehat\nabla$ and $\nabla$ on $E$.
	As a local basis of $\widetilde{\nabla}$-parallel vector fields on $\widetilde M$ descends to a local basis of sections of $E$ which are $\nabla$-parallel, we conclude that $\nabla$ is a flat connection. Moreover, as $Q$ descends from $\widetilde \omega$, which is $\widetilde \nabla$-parallel, it follows that $Q$ is $\nabla$-parallel. 
	
	The only thing remaining to be checked is that $\nabla$ satisfies the Griffiths transversality condition. For this, we need to obtain an explicit expression for $\nabla$. In fact, it is more convenient to first obtain an explicit expression for $\widehat \nabla$. Note that by Lemma \ref{lemma:derivative_chi} and Lemma \ref{lemma:LC_difference}, we have
	\begin{equation}
	\begin{split}
	&\phi^{1,0}(\nabla^{\widetilde g}_Y X)\\
	&=\chi(\nabla^{\widetilde g}_Y X)\Phi_L + \Phi_L \otimes \pi_*(\nabla_Y^{\widetilde g}X)^{1,0} \\
	&=\dif(\chi(X))(Y)\Phi_L -(\nabla^{\wt g}_Y \chi)(X) \Phi_L + \Phi_L \otimes \pi_*(\nabla_Y^{\widetilde g}X)^{1,0} \\
	&=\nabla^\chi_Y(\chi(X)\Phi_L) -\chi(X)\nabla^\chi_Y\Phi_L - (\nabla_Y^{\widetilde g}\chi)(X)\Phi_L + \Phi_L \otimes \pi_*(\nabla^{\widetilde g}_YX^{1,0})\\
	&= \nabla^\chi_Y(\chi(X)\Phi_L) +(-\chi(X)\chi(Y)+\chi(Y)\chi(X)+\pi^* h_M(Y,X))\Phi_L\\
	&\quad + \Phi_L \otimes \left(\nabla^{g_M}_{Y} \pi_* X^{1,0}
	+ \pi_*\left(\chi(Y)X^{1,0}+\chi(X)Y^{1,0}\right)\right)\\
	&=\nabla^\chi_Y(\chi(X)\Phi_L) +\pi^* h_M(Y,X)\Phi_L + \nabla_{Y}^{\chi,g_M}(\Phi_L \otimes X^{1,0}) + \chi(X)\Phi_L \otimes \pi_*Y^{1,0},
	\end{split}
	\end{equation}
	where $\nabla^{\chi, g_M}$ is the tensor product of $\nabla^{\chi}$ and $\nabla^{g_M}$.
	Thus, if we define a connection $\widehat \nabla^{1,0}$ on $L \oplus (L \otimes T^{1,0}M)$ to be given by
	\begin{equation}
	\widehat\nabla_Y^{1,0}(\sigma_1 + \sigma_2 \otimes Z)= \nabla_Y^{\chi}\sigma_1 + \nabla_Y^{\chi,g_M}(\sigma_2 \otimes Z) + h_M(Y,Z)\sigma_2 + \sigma_1 \otimes Y^{1,0},
	\end{equation}
	with $\sigma_1, \sigma_2 \in \Gamma(L)$ and $Z\in \Gamma(T^{1,0}M)$ being holomorphic sections, then we have
	\begin{equation}
	\phi^{1,0}(\nabla^{\widetilde g}_YX)=\widehat\nabla^{1,0}_Y(\phi^{1,0}(X)).
	\end{equation}
	This is therefore the $\phi^{1,0}$ part of the connection $\widehat \nabla$ induced on $E_{\C}$ by the Levi-Civita connection $\nabla^{\widetilde g}$.
	A similar computation on the $\phi^{0,1}$ part tells us that the full connection $\widehat \nabla$ is given by
	\begin{equation}
	\begin{split}
	&\widehat\nabla_Y(\sigma_1 + \sigma_2 \otimes Z + \ol \varsigma_2 \otimes \ol W + \ol \varsigma_1)\\
	&\qquad= \nabla_Y^{\chi}\sigma_1 +h_M(Y,Z)\sigma_2+ \nabla_Y^{\chi,g_M}(\sigma_2 \otimes Z) + \sigma_1 \otimes Y^{1,0}\\
	&\qquad\quad
	+ \nabla_Y^{\ol \chi,g_M}(\ol\varsigma_2 \otimes \ol W) + \ol\varsigma_1 \otimes Y^{0,1} + \nabla_Y^{\ol \chi}\ol \varsigma_1 + h_M(W,Y)\ol\varsigma_2.
	\end{split}
	\end{equation}
	We know that the holomorphic difference tensor $\wt \eta$ descends to the intrinsic deviance $\eta\colon L \otimes L \rightarrow \sharp_2 S_{3,0}M$. Using the Hermitian form $h_L^{-1}$ on $L^*$ to make an identification $L^* \cong \ol L$, we obtain maps
	\begin{equation}
	\begin{split}
	h_L^{-1} \circ \eta\in\Omega^{1,0}(\Hom_{\C}(L \otimes T^{1,0}M ,\ol L \otimes T^{0,1}M)),\\
	\ol h_L{}^{-1} \circ \ol \eta\in\Omega^{0,1}(\Hom_{\C}(\ol L \otimes T^{0,1}M , L \otimes T^{1,0}M)).
	\end{split}
	\end{equation}
	Thus, the Gau\ss--Manin connection $\nabla$ on $E_{\C}$ is given by
	\begin{equation}
	\begin{split}
	&\nabla_Y(\sigma_1 + \sigma_2 \otimes Z + \ol \varsigma_2 \otimes \ol W + \ol \varsigma_1)\\
	&\qquad= \widehat\nabla_Y(\sigma_1 + \sigma_2 \otimes Z + \ol \varsigma_2 \otimes \ol W + \ol \varsigma_1)\\
	&\qquad\quad + (h_L^{-1} \circ \eta)_Y(\sigma_2 \otimes Z) + (\ol h_L{}^{-1} \circ \ol \eta)_Y(\ol \varsigma_2 \otimes \ol W).
	\end{split}
	\end{equation}
	The fact that $\nabla$ satisfies the Griffiths transversality condition may now be made manifest by rearranging the terms of the above expression as
	\begin{align}
	&\nabla_Y(\sigma_1 + \sigma_2 \otimes Z + \ol \varsigma_2 \otimes \ol W + \ol \varsigma_1)\nonumber\\
	&\qquad= \nabla_Y^{\chi}\sigma_1 +h_M(Y^{0,1},Z)\sigma_2+ \nabla_Y^{\chi,g_M}(\sigma_2 \otimes Z) + \sigma_1 \otimes Y^{1,0}\nonumber\\
	&\qquad\quad
	+ \nabla_Y^{\ol \chi,g_M}(\ol\varsigma_2 \otimes \ol W) + \ol\varsigma_1 \otimes Y^{0,1} + \nabla_Y^{\ol \chi}\ol \varsigma_1 + h_M(W,Y^{1,0})\ol\varsigma_2\nonumber\\
	&\qquad\quad + (h_L^{-1} \circ \eta)_{Y^{1,0}}(\sigma_2 \otimes Z) + (\ol h_L{}^{-1} \circ \ol \eta)_{Y^{0,1}}(\ol \varsigma_2 \otimes \ol W)\\
	&\qquad=\nabla_Y^{\chi}\sigma_1 + \sigma_1 \otimes Y^{1,0}\nonumber\\
	&\qquad\quad + h_M(Y^{0,1},Z)\sigma_2 + \nabla_Y^{\chi,g_M}(\sigma_2 \otimes Z) + (h_L^{-1} \circ \eta)_{Y^{1,0}}(\sigma_2 \otimes Z)\nonumber\\
	&\qquad\quad + (\ol h_L{}^{-1} \circ \ol \eta)_{Y^{0,1}}(\ol \varsigma_2 \otimes \ol W) + \nabla_Y^{\ol \chi,g_M}(\ol\varsigma_2 \otimes \ol W) + h_M(W,Y^{1,0})\ol\varsigma_2 \nonumber\\
	&\qquad\quad + \ol\varsigma_1 \otimes Y^{0,1} + \nabla_Y^{\ol \chi}\ol \varsigma_1.\nonumber\qedhere
	\end{align}
	\end{proof}
	A physical interpretation of the decomposition \eqref{eq:HD} was given in \cite[Equation (2.21) (p.\ 328)]{BCOV}.
	Moreover, \eqref{eq:GM} is equivalent to the Picard--Fuchs equations described in \cite[Equation (1.4) (p.\ 283)]{PicFuchs}.
\begin{rmk}\label{rmk:TMtilde=[E]}
	By the isomorphism \eqref{eq:TMisoEpi}, we have also proven that $T\wt M\cong E_{\pi}$.
\end{rmk}
\begin{rmk}
On this VPHS, the Griffiths and Weil Hermitian forms are
\begin{align}
h_{\mathrm G}=-h_L + h_L \otimes h_M,
&\hfill&
h_{\mathrm W}=\ol{h}_L + h_L \otimes h_M.
\end{align}
As expected, the imaginary part of the two forms are the same, whereas the real part of $h_{\mathrm W}$ corresponds to the real part of $h_{\mathrm G}$ via a signature change that makes it positive definite.
\end{rmk}

Now that we know how to obtain variations of polarised Hodge structure from PSK structures, we next show the converse, i.e.\ that VPHS of the above type give rise to a PSK structure $\wt M \rightarrow M$.

\begin{prop}\label{pr:VPHS2}
	Let $M$ be a Kähler manifold with an integral Kähler form $\omega_M$ equipped with a VPHS $(E_{\C}=\bigoplus E^{p,q}, Q, \nabla)$ of weight $3$ of the form 
	\begin{equation}
		E_{\C} = L\oplus (L \otimes T^{1,0}M) \oplus (\ol{L} \otimes T^{0,1}M) \oplus \ol{L},
	\end{equation}
	where $L$ is a complex line bundle with curvature $-2\I \omega_M$.
	Furthermore, suppose $\nabla$ is real with respect to the canonical real structure and let the projection of $\nabla_Y \sigma$ onto $L \otimes T^{1,0}M$ be $\sigma \otimes Y^{1,0}$ for any sections $\sigma$ and $Y$ of $L$ and $TM$ respectively.
	Then, $M$ admits a PSK structure.
\end{prop}
\begin{proof}
	First of all, notice that by Griffiths transversality and the integrability of $I$ on $M$, the exterior covariant derivative associated to $\nabla$ splits as
	\begin{equation}
	\dif^{\nabla}\colon \Omega^{p,q}(E_{\C})\longrightarrow \Omega^{p+1,q}(E_{\C})\oplus \Omega^{p,q+1}(E),
	\end{equation}
	for all $p,q$.
	The flatness of $\nabla$ implies that the anti-holomorphic projection $\ol \partial_E\colon \Omega^{\bullet,\bullet}(E_{\C})\to \Omega^{\bullet,\bullet+1}(E_{\C})$ squares to zero, making $\ol \partial_E$ a Dolbeault operator on $E_{\C}$.
	Moreover, again by \eqref{eq:grif-trans}, the restriction of $\ol\partial_E$ to $L$-valued forms has image in $\Omega^{\bullet,\bullet}(L)$, and thus, it defines a Dolbeault operator $\ol\partial_L$ on $L$, giving it the structure of a holomorphic line bundle.
	Furthermore, the bilinear form $-2\I Q$ restricted to $\ol L \otimes L$ gives a (positive-definite) hermitian form $h_L$ on $L$.
	
	Let $\nabla^{\chi}$ be the linear connection on $L$ induced by restricting $\nabla$ to $L$.
	In other words, $\nabla^{\chi}_Y \sigma$ is the restriction of $\nabla_Y \sigma$ to $L$, for all $Y\in\field{M}$ and $\sigma\in\Gamma(L)$.
	By definition, $\nabla^\chi$ is compatible with the holomorphic structure on $L$, and, since $\nabla$ preserves $Q$, the connection $\nabla^{\chi}$ preserves $h_L$.
	In fact, for $\sigma_1,\sigma_2\in\Gamma(L)$ we have
	\begin{align}
	\begin{split}
	&\dif(h_L(\ol\sigma_1,\sigma_2))
	=-2\I\, \dif(Q(\ol\sigma_1,\sigma_2))
	=-2\I Q(\ol{\nabla_Y \sigma_1},\sigma_2))-2\I Q( \ol\sigma_1,\nabla_Y\sigma_2)\\
	&\quad=h_L(\ol{\nabla^{\chi}_Y \sigma_1},\sigma_2)+h_L( \ol\sigma_1,\nabla^{\chi}_Y\sigma_2)
	=h_L(\nabla^{\chi}_Y \ol\sigma_1,\sigma_2)+h_L( \ol\sigma_1,\nabla^{\chi}_Y\sigma_2).
	\end{split}
	\end{align}
	The connection $\nabla^{\chi}$ is thus the Chern connection on $L$.
	From the hypotheses, we also have
	\begin{equation}
		\nabla_Y \sigma = \nabla^\chi_Y \sigma + \sigma \otimes Y^{1,0}.
	\end{equation}
	We now take $\pi\colon \wt M \rightarrow M$ to be $L\rightarrow M$ with the zero section removed and show that it has all the requisite structure. 
Note that $\wt M$ can be thought of as the frame bundle for $L$ and so it is a principal $\C^\times$-bundle associated to the line bundle $L$. Let $\zeta$ be its fundamental vector field.
	Moreover, the Chern connection $\nabla^\chi$ on $L$ gives a principal connection $\chi$ on the principal $\C^\times$-bundle $\wt M$ so that
	\begin{equation}
		\nabla^\chi_Y \Phi_L = \chi(Y)\Phi_L.
	\end{equation}
	 Consider the map $\phi\colon T\wt M \rightarrow E_{\pi}$
	 given by the explicit formula
	\begin{equation}
	\begin{split}
	&\phi(X) = \chi(X)\Phi_L + \Phi_L \otimes \pi_*X^{1,0} + \ol\Phi_L \otimes \pi_*X^{0,1} + \ol \chi(X)\ol \Phi_L.
	\end{split}
	\end{equation}
	This has an inverse
	$\phi^{-1}\colon E_{\pi}\rightarrow T\wt M$ 
	given by 
	\begin{equation}
	\begin{split}
	&\phi^{-1}(f\Phi_L + \Phi_L \otimes Y^{1,0} + \ol \Phi_L \otimes Y^{0,1} + \ol f \ol \Phi_L)\\
	&\qquad\qquad\qquad=f\zeta + (Y^{1,0})^\chi + (Y^{0,1})^{\ol \chi} + \ol f \ol \zeta.
	\end{split}
	\end{equation}
	Thus, we have an identification of vector bundles $T\wt M \cong E_{\pi}$ over $\wt M$.
	Under this identification, the bilinear form $Q$ and connection $\nabla$ define a $2$-form $\wt \omega$ on $\wt M$ and a connection $\wt \nabla$ preserving it. Moreover, we have
	\begin{equation}
	\begin{split}
		\wt \nabla_X \zeta &= \phi^{-1}(\nabla_X \Phi_L) = \phi^{-1}(\nabla^\chi_X \Phi_L + \Phi_L\otimes \pi_* X^{1,0}) \\
		&=\phi^{-1}(\chi(X)\Phi_L + \Phi_L\otimes \pi_* X^{1,0})=\chi(X)\zeta + (\pi_*X^{1,0})^\chi = X^{1,0}.
	\end{split}
	\end{equation}
	A similar argument shows $\wt \nabla_X\ol \zeta =X^{0,1}$. 

	Finally, $\wt M$ can be given the following hermitian structure
	\begin{align}
		\wt h =\wt g+\I\wt\omega= q(\pi^*h_M-\ol \chi \otimes \chi),&\hfill &q\colonequals h_L(\ol\Phi_L, \Phi_L).
	\end{align}
	In order to simplify the imminent calculations, we compute
	\begin{equation}
	\dif q
	=h_L(\ol{\nabla^{\chi}\Phi_L},\Phi_L)+h_L(\ol\Phi_L,\nabla^{\chi}\Phi_L)
	=h_L(\ol\Phi_L,\Phi_L)(\ol\chi+\chi)
	=q(\ol\chi+\chi).
	\end{equation}
	Therefore, we obtain
	\begin{align}
	\begin{split}
	\dif \wt\omega
	&=\dif(q(\pi^*\omega_M+\frac{i}{2}\,\ol \chi \wedge \chi))\\
	&=q(\ol\chi+\chi)\wedge (\pi^*\omega_M+\frac{i}{2}\,\ol \chi \wedge \chi)
	+q(-\pi^*\omega_M\wedge \chi-\pi^*\omega_M\wedge \chi)
	=0,\\
	\Lie{\zeta}\wt g
	&=\Lie{\zeta}(q(\pi^*g_M-\ol\chi\chi))
	=\Lie{\zeta}q(\pi^*g_M-\ol\chi\chi)
	=\iota_{\zeta}\dif q(\pi^*g_M-\ol\chi\chi)
	=\wt g,\\
	\dif\iota_{\zeta} \wt g
	&=-\frac{1}{2}\,\dif (q\ol\chi)
	=-\frac{1}{2}(\dif q\wedge\ol\chi+q\,\dif\ol\chi)
	=-\frac{1}{2}\,q(\chi\wedge\ol\chi+2\I\pi^*\omega_M)
	=-\I\wt\omega.
	\end{split}
	\end{align}
	Using the version of the Koszul formula in terms of the Lie derivative, we obtain
	\begin{align}
	2\wt g(\nabla^{\wt g} \zeta,\cdot)
	&=\Lie{\zeta}\wt g+\dif\iota_\zeta \wt g
	=\wt g-\I\wt\omega
	=\wt g(\id-\I I,\cdot)
	=2\wt g(\pi^{1,0},\cdot).
	\end{align}
	This implies $\nabla^{\wt g}\zeta=\pi^{1,0}$, and similarly we obtain $\nabla^{\wt g}\ol\zeta=\pi^{0,1}$.
	
	This then gives $\wt M$ the structure of a CSK manifold.
	Hence, we have a PSK structure $\wt M \rightarrow M$.	
\end{proof}

	The VPHS $(E_{\C}=\bigoplus E^{p,q}, Q,\nabla)$ is completely determined by the choice of a Kähler structure $( g_M, I,\omega_M)$ on $M$, a line bundle $L\rightarrow M$ with Hermitian form $h_L$ such that the curvature of its Chern connection is $-2\I\omega_M$, and an intrinsic deviance $\eta \in\Gamma( L^* \otimes L^* \otimes \sharp_2 S_{3,0}M)$. The intrinsic deviance goes into the data of the Gau\ss--Manin connection $\nabla$ through \eqref{eq:GM}. The flatness condition then is equivalent to the following equations:
	\begin{align}\label{eq:D1_D2}
		R^{g_M}+R_{\Pj^n_{\C}}+\langle h^{-1}_L \otimes h^{-1}_L,[\eta\wedge\ol{\eta}]\rangle=0,
		&\hfill&
		\dif^{\chi,g_M}\eta = 0.
	\end{align}
	Here, $\langle h^{-1}_L \otimes h^{-1}_L,[\eta\wedge\ol{\eta}]\rangle$ is to be interpreted as $(h_L^{-1}\otimes h_L^{-1}\otimes [\cdot\wedge\cdot])(\ol \eta\otimes\eta)$,
	\begin{equation}
	\dif^{\chi,g_M}\colon \Omega^1(M,L^* \otimes L^* \otimes T^{0,1}M\otimes T^*_{1,0}M)\longrightarrow \Omega^2 (M,(L^* \otimes L^* \otimes T^{0,1}M\otimes T^*_{1,0}M)
	\end{equation}
	is the exterior covariant derivative with respect to $\nabla^{\chi,g_M}$, and $R_{\Pj_{\C}^n}$ is formally the Riemann curvature of the complex projective space.
	This is more or less the characterisation of PSK manifolds given in of \cite[Theorem 7.6 (p.\ 21)]{MMPSK2019}. Strictly speaking, the equations \textbf{D1} and \textbf{D2} appearing there are the pullback of the above equations along some section of $L$ with unit norm.

\begin{eg}\label{eg:CHn}
	Let $M$ be the complex hyperbolic space $\Hy_\C^n$ (for a detailed definition see e.g.\ \cite[\S 9]{MMPSK2019}) and let $L$ be its tautological line bundle $\mathcal O(-1)$ (not to be confused with tautological pullback bundles). The curvature $R^{g_M}$ is then $-R_{\Pj^n_{\C}}$, leaving us with the following equation for the intrinsic deviance $\eta$:
	\begin{align}
	\langle h^{-1}_L \otimes h^{-1}_L,[\eta\wedge\ol{\eta}]\rangle=0,
	&\hfill&
	\dif^{\chi,g_M}\eta = 0.
	\end{align}
	Notice that $\eta = 0$ is a solution of the above (in fact the only one, see e.g.\ \cite[Proposition 9.3 (p.\ 31)]{MMPSK2019}). The resulting special connection $\nabla$ pulls back to $E_\pi\cong T\wt M$ to give a connection $\wt \nabla$ that coincides with the Levi-Civita connection associated to the standard flat indefinite metric of signature $(2n,2)$ on $\wt M \subset \C^{n+1}$.
	
	Let us work this out more explicitly in terms of the standard coordinate chart $z=(z_0,z_1 ,\dots,z_n)$ on $\C^{n+1}$. We have the following:
	\begin{equation}
	\begin{split}
		&\wt M = \{(z_0,z_1,\ldots,z_n) \in \C^{n+1}\mid 0< |z_1|^2 + \cdots + |z_n|^2 <|z_0|^2 \},\\
		&\wt g = \sum_{i=1}^n |\dif z_i|^2-|\dif z_0|^2, \quad \zeta = \sum_{i=0}^n z_i\partial_{z_i},
		\quad \chi = \frac{\ol z_0\dif z_0-\sum_{i=1}^n \ol z_i\dif z_i}{|z_0|^2-\sum_{i=1}^n |z_i|^2}.
	\end{split}
	\end{equation}
	The special connection $\wt \nabla = \nabla^{\wt g}$ on $T\wt M$ admits a basis of globally defined parallel sections $\{\partial_{z_i}, \partial_{\ol z_i}\}_{i=0}^n$.
	This induces an explicit basis of sections parallel with respect to the connection $\nabla$ on $E$ (which characterises the connection).
	
	First of all, let us note that since $z_0$ is nowhere vanishing on $\wt M$, we can define a global complex coordinate chart $X=(X_i)_{i=1}^n=(z_i/z_0)_{i=1}^n$ on $M=\wt M/\C^\times$. Furthermore, the section $z_0^{-1}\Phi_L$ of $L_\pi$ may be seen to be the pullback of some section of $L$.
	Let us define a section $\sigma$ of $E_\C$ so that
	\begin{equation}
		\sigma_\pi
		= \frac{z_0^{-1}}{\sqrt{1-\sum_{i=1}^n|X_i|^2}}\,\Phi_L
		= \frac{z_0^{-1}}{\sqrt{1-\norm{X}^2}}\,\Phi_L.
	\end{equation}
	The point of including the extra factor before $\Phi_L$ is so that we have $h_L(\ol \sigma,\sigma)=1$.
	In terms of the section $\sigma$ and coordinates $X_i$, the $\wt \nabla$-parallel sections $\{\partial_{z_i}, \partial_{\ol z_i}\}_{i=0}^n$ of $T\wt M_{\C}$ correspond to the following $\nabla$-parallel sections of $E_{\C}$:
	\begin{equation}
		\begin{split}
		&s_0= \frac{1}{\sqrt{1-\norm{X}^2}}\,\sigma - \sqrt{1-\norm{X}^2}\,\sigma \otimes \sum_{j=1}^n X_j\partial_{X_j},\\
		&s_{j}=-\frac{\ol X_j}{\sqrt{1-\norm{X}^2}}\,\sigma + \sqrt{1-\norm{X}^2}\,\sigma \otimes \partial_{X_j},\qquad j=1,\dots,n,\\
		&\ol s_0= \frac{1}{\sqrt{1-\norm{X}^2}}\,\ol \sigma - \sqrt{1-\norm{X}^2}\,\ol \sigma \otimes \sum_{j=1}^n\ol X_j\partial_{\ol X_j},\\
		&\ol s_{j}=- \frac{X_j}{\sqrt{1-\norm{X}^2}}\,\ol\sigma + \sqrt{1-\norm{X}^2}\,\ol\sigma \otimes \partial_{\ol X_j},\qquad j=1,\dots, n.
		\end{split}
			\end{equation}
	Notice that these sections satisfy $Q(\ol s_i, s_j) =\frac{1}{2\I}(-1)^{\delta_{i,0}} \delta_{ij}$, and for the Griffiths hermitian form we have
	\begin{equation}
		h_{\mathrm{G}}(\ol s_i, s_j)
		=Q(\ol s_i,I_{\mathrm{G}} s_j)+\I Q(\ol s_i,s_j)
		=2\I Q(\ol s_i, s_j)
		=(-1)^{\delta_{i,0}} \delta_{ij}.
	\end{equation}		
	
	We can finally write the K\"ahler structure on $M$ with respect to the coordinate system $X=(X_i)_{i=1}^n$, obtaining
	\begin{align}\label{eq:h_M_Hyp}
	\begin{split}
	h_M
	&=\frac{\wt h}{-\wt h(\ol\zeta,\zeta)}+\ol\chi\otimes \chi\\
	&=\frac{\sum_{i=1}^n\dif\ol z_i\otimes\dif z_i-\dif\ol z_0\otimes \dif z_0}{|z_0|^2-\sum_{k=1}^n |z_k|^2}\\
	&\quad+\frac{(z_0\dif \ol z_0-\sum_{i=1}^n z_i\dif \ol z_i)\otimes (\ol z_0\dif z_0-\sum_{j=1}^n \ol z_j\dif z_j)}{(|z_0|^2-\sum_{k=1}^n |z_k|^2)^2}\\
	&=\sum_{i=1}^n\frac{\dif\ol X_i\otimes\dif X_i}{1-\norm{X}^2}
	+\sum_{i,j=1}^n \frac{X_i\dif \ol X_i\otimes \ol X_j\dif X_j}{\left(1-\norm{X}^2\right)^2}.
	\end{split}
	\end{align}
\end{eg}

\section{The twist construction}\label{sec:twist}
In this section, we will show that under proper circumstances, the twist construction reduces to a quotient by the action of a finite cyclic group (Lemma \ref{lemma:twist}).
In addition, we explicitly determine how the pushforward of a vector field through the quotient map relates to its twist (Lemma \ref{lem:tw_symmetries}).

The \emph{twist construction} was introduced by Swann in order to unify and generalise several differential-geometric constructions arising from T-duality in physics. 
The construction takes as input the following \emph{twist data} on a manifold $W$: a vector field $Z$ generating a $\U(1)$-action, an integral closed $2$-form $\omega$ with respect to which $Z$ is Hamiltonian, and a choice of a Hamiltonian function $f$ that is nowhere vanishing.
Its output is another manifold $W'$ with a $\U(1)$-action, together with a bijective correspondence between $\U(1)$-invariant tensor fields on $W$ and $W'$.

Very roughly, the way this is achieved is by building a principal $\U(1)$-bundle $P \rightarrow W$ on top of $W$, lifting the given $\U(1)$-action on $W$ to a $\U(1)$-action on $P$ (different from the principal $\U(1)$-action), and then quotienting $P$ by this lifted $\U(1)$-action.
The principal $\U(1)$-action on $P$ then descends to the quotient $W'$.

\begin{prop}[\cite{Swann2010} Propositions 2.1, 2.3, \cite{MaciaSwann2015} \S 4]
	Given twist data $(Z,\omega,f)$, there is a principal $\U(1)$-bundle $p^P_W\colon P\rightarrow W$ with principal connection $\theta_P$ having curvature $\omega$ and fundamental vector field $X_P$ such that the lift 
	\begin{equation}
	Z_P = Z^{\theta_P} + f X_P
	\end{equation}
	of $Z$, generates a $\U(1)$-action on $P$ and so gives a well-defined quotient map
	\begin{equation}
	p^P_{W'}\colon P \longrightarrow W' \colonequals P/\langle Z_P \rangle,
	\end{equation}
	such that the principal $\U(1)$-action on $P$ descends to a $\U(1)$-action on $W'$.
	Furthermore, the $\theta_P$-horizontal lift of any $Z$-invariant vector field on $W$ and the pullback of any $Z$-invariant function on $W$ to $P$ descend to a well-defined $\U(1)$-invariant vector field and a well-defined $\U(1)$-invariant function on $W'$ respectively.
\end{prop}
The twist correspondence assigns to a $Z$-invariant vector field $X$ on $W$ the well-defined vector field $X'$ on $W'$ such that $(p^{P}_{W'})_{*} X^{\theta_P}=X'$.
The correspondence also assigns to a $Z$-invariant function $h$ on $W$ a unique function $h'$ on $W'$ such that the pullbacks of $h$ and $h'$ to $P$ agree.
We call $X'$ and $h'$ \emph{twists} of $X$ and $h$ respectively. We write this as
\begin{align}
	\tw(X)= X',& \hfill & \tw(h) =h'.
\end{align}
By stipulating compatibility with tensor products and contractions, this map can be extended to arbitrary $\U(1)$-invariant tensor fields. In particular, for differential forms $\alpha$ on $W$ and $\alpha'$ on $W'$, we have $\tw(\alpha) = \alpha'$ if and only if $(p^P_{W})^*\alpha-(p^P_{W'})^*\alpha'$ vanishes on $\theta_P$-horizontal vector fields.

Since we aim to show that for a certain class of twist data we can obtain the twisted manifold as a quotient by a discrete group, we will briefly recall some facts on this matter.

Given a principal $G$-bundle $\pi_S\colon S\to B$, and a normal subgroup $H \subseteq G$, the quotient $S/H$ has an induced structure of a principal $G/H$-bundle.

We specialise to $G=\U(1)$ and $H =\Z_k$, regarded as the group of $k$-th roots of unity in $\U(1)$. In this case, the quotient $\U(1)/\Z_k$ is isomorphic to $\U(1)$ itself via the isomorphism $[x]\mapsto x^k$. So, $S_k\colonequals S/\Z_k$ is also a principal $\U(1)$-bundle $\pi_k\colon S_k\to B$ with the $\U(1)$-action on it given by
\begin{equation}
	q(u)\cdot e^{\I t}
	= q(u \cdot e^{\I t/k}). 
\end{equation}
Here, $u$ is a point in $S$ and $q\colon S \rightarrow S_k$ is the quotient map.

Notice also that a principal connection $\varphi$ on $S$ induces a unique principal connection $\varphi_k$ on $S_k$, that is, it has the property that $q_*$ maps horizontal vector fields to horizontal ones.
We now want to compare $\varphi$ and $\varphi_k$, and the fundamental vector fields $Z$ and $Z_k$ on these two principal $\U(1)$-bundles using the differential of the quotient map $q$.
Recall that the fundamental vector fields $Z$ on $S$ and $Z_k$ on $S_k$ are given by
\begin{align}
	Z_u =\left.\frac{\dif}{\dif t}(u\cdot e^{\I t})\right|_{t=0},& \hfill&
	(Z_k)_{q(u)} =\left.\frac{\dif}{\dif t}(q(u)\cdot e^{\I t})\right|_{t=0}.
\end{align}
Therefore, under the action of the differential $q_*$, we have
\begin{equation}
	\begin{split}
	(q_*Z)_u 
	&= \left.\frac{\dif}{\dif t}\,q(u\cdot e^{\I t})\right|_{t=0}
	=\left.\frac{\dif}{\dif t}(q(u)\cdot e^{\I k t})\right|_{t=0}
	=k\left.\frac{\dif}{\dif t'}(q(u)\cdot e^{\I t'})\right|_{t'=0}\\
	&=k(Z_k)_{q(u)}.
	\end{split}
\end{equation}
In conclusion, we get
\begin{equation}\label{eq:fundamentalVF_principalConn_discrete_quotient}
q_*Z = k(Z_k)_q,
\qquad
q^*\varphi_k = k\varphi,
\qquad
q_*Y^\varphi=(Y^{\varphi_k})_q \textrm{ for }Y\in\field{B},
\end{equation}
where the second equation follows from $\varphi(Z)=1=\varphi_k(Z_k)$, and the third from
\begin{equation}
\varphi_k (q_*(Y^{\varphi}))
=q^*(\varphi_k)(Y^{\varphi})
=k\varphi(Y^{\varphi})
=0.
\end{equation}

In order to simplify the following statements, we extend the definition of $S_k$ for $k=0$ by declaring it to be the trivial bundle 
\begin{align}
S_0\colonequals B\times \U(1),&\hfill&
\pi_0=\mathrm{pr}_{B}\colonequals B\times \U(1)\longrightarrow B.
\end{align}
Moreover, we also extend the objects
\begin{equation}\label{eq:objects_S0}
q(s)=(\pi_S(s),1)\in S_0,
\qquad Z_0=\partial_t,
\qquad \varphi_0=\dif t,
\end{equation}
where $t=t_{\pi_0}$ is (the pullback of) the standard coordinate chart on $\U(1)$ such that the elements of $\U(1)$ are $e^{\I t}$.
Notice that the equalities \eqref{eq:fundamentalVF_principalConn_discrete_quotient} still hold and $q_*$ still preserves horizontal vector fields.

We now consider the following instance of the twist construction, that we will make use of later.
\begin{lemma}\label{lemma:twist}
	Let $\pi_S\colon S \rightarrow B$ be a principal $\U(1)$-bundle with connection $1$-form $\varphi$ and fundamental vector field $Z$. Let $k$ be a non-negative integer, and let $f$ be a nowhere vanishing function and $\beta$ be a $1$-form, both defined on $B$. Then the twist of $S$ with respect to the twist data
	\begin{equation}\label{eq:twist_data}
	\begin{split}
	(Z, \omega, \pi_S^*f)
	:\!&= (Z, \dif((\pi_S^*f+k)\varphi + \pi_S^*\beta),\pi_S^*f)\\
	&=(Z, \dif((f+k)\varphi + \pi_S^*\beta),f)
	\end{split}
	\end{equation}
	is the principal bundle $\pi_k\colon S_k\to B$.
	Moreover,
	\begin{equation}
	\begin{split}
	&\tw(\pi_S^*\alpha) = \pi_{k}^*\alpha, \qquad\qquad \tw(\varphi) = - \frac{1}{f}(\varphi_k + \pi_k^*\beta), \\
	 &\tw(X^\varphi)= X^{\varphi_k} - \beta(X)Z_k, \qquad\qquad \tw(Z)= - f Z_k,
	\end{split}
	\end{equation}
	for all differential forms $\alpha$ and vector fields $X$ on $B$.
\end{lemma}

\begin{proof}
	As $\omega$ is exact, the principal $\U(1)$-bundle $P$ may be taken to be the trivial bundle $ \mathrm{pr}_S\colon S \times \U(1) \rightarrow S$.
	The connection form $\theta_P$ and the fundamental vector field $X_P$ may then be taken to be 
	\begin{align}
	\theta_P = \mathrm{pr}_S^*((f+k)\varphi + \pi_S^*\beta)+\dif \tau,&
	\hfill&
	X_P = \partial_{\tau},
	\end{align}
	where $\tau$ is the standard coordinate chart on $\U(1)$.
	Since $P$ is a product space, we have an identification 
	\begin{equation}
		TP \cong (TS)_{\mathrm{pr}_S} \oplus (T\, \U(1))_{\mathrm{pr}_{\U(1)}}.
	\end{equation}
	In terms of this identification, we may write the horizontal lift $Z^{\theta_P}$ of $Z$ as 
	\begin{equation}
	Z^{\theta_P}
	= Z - \theta_P(Z)\partial_{\tau}
	= Z - (f + k)\partial_{\tau}.
	\end{equation}
	Thus, the twisted lift $Z_P$ is given by
	\begin{equation}
	Z_P
	= Z^{\theta_P} + fX_P
	= Z - k\partial_{\tau}.
	\end{equation}
	Now we see that our choice of bundle $P$ was the correct one, since $Z_P$ generates a $\U(1)$-action on $P=S\times \U(1)$, namely
	\begin{equation}
	\left(s,u\right)\cdot e^{\I \tau}
	= \left(s\cdot e^{\I \tau}, u e^{-\I k\tau}\right).
	\end{equation}
	We now define a map $\varpi_k\colon P\rightarrow S_k$ by $\varpi_k(s,u)\colonequals q(s)\cdot u$ and show that it is a quotient map for the $\U(1)$-action generated by $Z_P$.
	When $k=0$, the quotient is a trivial bundle, and its quotient map is $\pi_{S}\times\id_{\U(1)}=\varpi_0$.
	When $k\neq 0$, we can say that $\varpi_k$ is surjective and $\U(1)$-invariant by the equation
	\begin{equation}
	\varpi_k((s,u)\cdot e^{\I \tau})
		=q(s\cdot e^{\I \tau})\cdot u e^{-\I k\tau}
		=q(s)\cdot u
		=\varpi_k(s,u).
	\end{equation}
	Moreover, the $\U(1)$-action is transitive on the fibres of $\varpi_k$.
	Suppose now that $\varpi_k(s,u)=\varpi_k(s',u')$, and consider the following commutative diagram.
	\begin{equation}\label{diag:projections_twist}
	\begin{tikzcd}
		& S \times \U(1) \arrow[r, "\mathrm{pr}_S"] \arrow[d, "\varpi_k"']&S\arrow[d, "\pi_S"]\\
		&S_k \arrow[r, "\pi_k"] &B
	\end{tikzcd}
	\end{equation}
	The points $s$ and $s'$ belong to the same fibre of $\pi_S$, and so $s'=s\cdot e^{\I\tau}$.
	Hence,
\begin{align}
	\begin{split}
	q(s)\cdot u
	&=\varpi_k(s,u)
	=\varpi_k(s\cdot e^{\I \tau},u')
	=q(s\cdot e^{\I \tau})\cdot u'
	=q(s)\cdot e^{\I k\tau}u',
	\end{split}
\end{align}
	and since the $\U(1)$-action on $S_k$ is principal, we infer that $u'=ue^{-\I k\tau}$, so
	\begin{equation}
	(s',u')
	=(s\cdot e^{\I \tau},u e^{-\I k\tau})
	=(s,u)\cdot e^{\I \tau}.
	\end{equation}
	Thus, $\varpi_k$ is a quotient map, and $\pi_k\colon S_k\to B$ is the twisted bundle.
	
	Next, we describe the twist correspondence explicitly. It will be convenient to do this first for differential forms and then use the compatibility with contractions to obtain the correspondence for vector fields.
	
	By \eqref{diag:projections_twist}, we have $\mathrm{pr}_S^*\pi^*_S\alpha = \varpi_k^*\pi^*_{k}\alpha$ for any form $\alpha$ on $B$.
	So, we automatically have $\tw(\pi^*_S\alpha) = \pi^*_{k}\alpha$.

In order to compute $\tw(\varphi)$, observe that $\varpi_k$ is the composition of $q\times\id_{\U(1)}$ with the action map $\rho\colon S_k\times\U(1)\to S_k$, and moreover $\rho^*\varphi_k=\mathrm{pr}_{S_k}^*\varphi_k+\mathrm{pr}_{\U(1)}^*\dif \tau$ (cf.\ \cite[Example 27.4 (p.\ 326)]{TuDG}).
Therefore, we have
\begin{equation}
\varpi_k^*\varphi_k
=\mathrm{pr}_S^*q^*\varphi_k+\mathrm{pr}_{\U(1)}^*\dif \tau
=k\,\mathrm{pr}_S^*\varphi+\mathrm{pr}_{\U(1)}^*\dif \tau.
\end{equation}
Consider now the $Z_P$-invariant $1$-form
	\begin{equation}
	\mathrm{pr}_S^*\varphi - \frac{1}{f}\theta_P
	= -\frac{1}{f}(k\,\mathrm{pr}_S^*\varphi +\mathrm{pr}_{\U(1)}^*\dif \tau + \mathrm{pr}_S^*\pi_S^*\beta)
	=\varpi_k^*\left(-\frac{1}{f}(\varphi_k+\pi_k^*\beta)\right).
	\end{equation}
	This agrees with $\mathrm{pr}_S^*\varphi$ on $\theta_P$-horizontal vector fields, yielding
	\begin{equation}
	\tw(\varphi)=-\frac{1}{f}(\varphi_k+\pi_k^*\beta).
\end{equation}

	The twist correspondence for vector fields is determined by the fact that contractions with $1$-forms are preserved under the twist correspondence. We may thus check that
	\begin{align}
	&(\pi^*_{S}\alpha)(X^{\varphi})=\alpha(X)=(\pi^*_{k}\alpha)(X^{\varphi_k} - \beta(X)Z_k),\\
	&\varphi(X^{\varphi})=0=-\frac{1}{f}(\varphi_k(X^{\varphi_k}-\beta(X)Z_k) + (\pi_{k}^*\beta)(X^{\varphi_k}-\beta(X)Z_k)),\nonumber\\
	&(\pi^*_{S}\alpha)(Z)=0=(\pi^*_{k}\alpha)(-fZ_k),\nonumber\\
	&\varphi(Z) = 1= -\frac{1}{f}(\varphi_k(-fZ_k) + (\pi^*_{k}\beta)(-fZ_k)).\nonumber\qedhere
	\end{align}
\end{proof}
\begin{rmk}\label{rmk:rational_twist}
If $k\in\R$ is arbitrary, one could still construct the bundle $P$ so that the $\U(1)$-action generated by $Z$ lifts to a $\U(1)$-action on $P$ (generated by $Z_P$), but that bundle wouldn't be trivial in general \cite{Swann2010}.
In case $k=\frac{a}{b}\in\Q$, with $a,b\in\Z\smallsetminus \{0\}$, when $\frac{1}{b}\dif\varphi$ corresponds to an integral class, we have a $\U(1)$-bundle $S_{\frac{1}{b}}$ such that $S$ is a $\Z_b$-quotient of $S_{\frac{1}{b}}$.
If we now interpret $S_{\frac{1}{b}}$ as a principal $\Z_{b}$-bundle over $S$, we can define $P$ as the associated bundle $S_{\frac{1}{b}}\times_{\Z_{b}} \U(1)$.
This exhibits $P$ as a principal $\U(1)$-bundle over $S$.
Furthermore, $Z_P$ generates a $\U(1)$-action on $P$ given by
\begin{equation}
[s,u]\cdot e^{\I \tau}=[s e^{\I \tau},u e^{-\I a \tau}].
\end{equation}
With respect to this action, the orbit containing an arbitrary element $[s,e^{\I \tau}]$, also contains the elements of the form $[s e^{\frac{\I\tau}{a}},1]$.
Notice that we have the embedding 
\begin{align}
S_{\frac{1}{b}}\longrightarrow P,
&\hfill&
s\longmapsto [s,1].
\end{align}
The image of this embedding then intersects the orbits of the action generated by $Z_P$ in exactly $a$ points, on which the $Z_P$-action induces a $\Z_a$-action.
Hence, the quotient $P/\langle Z_P\rangle$ can be identified with the $\Z_a$-quotient of $S_{\frac{1}{b}}$, denoted $S_{\frac{a}{b}}=:S_k$.
\end{rmk}

\begin{lemma}\label{lem:tw_symmetries}
	Let $\pi_S\colon S \rightarrow B$ be a principal $\U(1)$-bundle with connection $1$\nobreakdash-form $\varphi$ and fundamental vector field $Z$.
	Let $(Z, \omega, f)$ and $\beta$ be as in Lemma \ref{lemma:twist}, and furthermore assume that $\omega$ is non-degenerate.
	Let $Y$ be a vector field on $S$ satisfying
	\begin{align}
		\Lie{Y}f = 0,&
		\hfill &
		\Lie{Y}\varphi= \Lie{Y}\pi_S^*\beta= 0. 
	\end{align}
	Then, $Y$ is $\omega$-Hamiltonian with Hamiltonian function 
	\begin{equation}
		f_Y=((f+k)\varphi + \pi_S^*\beta)(Y).
	\end{equation} 
	Moreover, $Y$ is $Z$-invariant and its twist is given by
	\begin{equation}\label{eq:twist_as_quotient}
		q_*Y = (\tw(Y) + f_Y Z_k)_q.
	\end{equation}
\end{lemma}
\begin{proof}
	By the Cartan formula, we have
	\begin{equation}
	\begin{split}
		\iota_{Y}\omega &= \iota_Y\dif((f+k)\varphi + \pi^*_S \beta)=(\Lie{Y}- \dif\circ\iota_Y)((f+k)\varphi + \pi^*_S \beta)\\
		&= -\dif\big(((f+k)\varphi + \pi^*_S \beta)(Y)\big)
		=-\dif f_Y.
		\end{split}
	\end{equation}
	Thus, $Y$ is $\omega$-Hamiltonian with Hamiltonian function $f_Y$.
	 
	Since the Lie derivative commutes with the exterior derivative, we also have
	\begin{equation}
	\begin{split}
		\Lie{Y}\omega
		&=\Lie{Y}\dif((f+k)\varphi + \pi_S^*\beta)\\
		&=\dif\circ\Lie{Y}((f+k)\varphi + \pi_S^*\beta)
		=0,\\
		\Lie{Y}(\iota_Z\omega)
		&=-\Lie{Y}\dif f 
		=-\dif \circ \Lie{Y} f
		=0.
		\end{split}
	\end{equation}
	In particular, we get 
	\begin{equation}
	\omega(\Lie{Y}Z, \cdot)=\Lie{Y}(\iota_Z\omega)-(\Lie{Y}\omega)(Z,\cdot)=0.
	\end{equation}
	Since $\omega$ is assumed to be non-degenerate, this implies $\Lie{Z}Y=-\Lie{Y}Z=0$.
	
	In order to take its twist, we decompose $Y$ into its vertical and $\varphi$-horizontal parts.
	Since $Y$ is $Z$-invariant, we can find a vector field $Y_B$ on $B$ such that
	\begin{equation}
		Y = Y_B^\varphi + \varphi(Y)Z.
	\end{equation}
	Then, by Lemma \ref{lemma:twist} we have
	\begin{equation}
	\begin{split}
	\tw(Y) &= \tw(Y_B^\varphi) + \tw(\varphi(Y)Z) = Y_B^{\varphi_k} - \beta(Y_B)Z_k - f\varphi(Y)Z_k\\
	&=Y_B^{\varphi_k} + k\varphi(Y)Z_k - ((f+k)\varphi + \pi_S^*\beta)(Y)Z_k\\
	&=Y_B^{\varphi_k} + k\varphi(Y)Z_k - f_YZ_k.
	\end{split}
	\end{equation}
	Note that the pullback of the vector field $Y_B^{\varphi_k} + k\varphi(Y)Z_k$ along the quotient map $q$ is precisely $q_*Y$. The statement to be proved thus follows.
\end{proof}

\section{The geometry of the c-map}\label{sec:c-map}
In this section, we prove the main theorem of this paper (Theorem \ref{th:sugra-c}), which describes the supergravity c-map as a natural construction on variations of Hodge structure.
Using this description, we furthermore construct the vertical Killing vector fields on the resulting quaternionic K\"ahler manifold (Proposition \ref{prop:vertical_Killing}), and show that they form the Heisenberg algebra (Corollary \ref{cor:Heisenberg}).
Finally, in \S \ref{ssec:functoriality}, we use the naturality of our c-map construction to interpret it as a functor from the category of certain variations of Hodge structure to the one of quaternionic K\"ahler manifolds.
As a last result (Proposition \ref{prop:vorsicht-funktor}), we show that our lifting of isomorphisms reproduces the one of infinitesimal automorphisms described in \cite{CST}.
\subsection{Rigid c-map}

The rigid c-map assigns a hyperkähler manifold to an affine special Kähler manifold $\wt M$. The hyperkähler structure may be defined on either the tangent or cotangent bundle. The two descriptions are related by the symplectic form $\wt \omega$ interpreted as a map $T\wt M \rightarrow T^*\wt M$. Note that, since $\wt \omega$ is $\wt \nabla$-parallel, the differential $\wt \omega_*$ preserves the splitting of $TT\wt M$ induced by $\wt \nabla$. In terms of the splitting, we have 
\begin{equation}
\begin{split}
	\wt \omega_*=
	\begin{pmatrix}
	\id & 0 \\ 0& \widetilde{\omega}
	\end{pmatrix}\colon \left(T\wt M \oplus T\wt M\right)_{T\wt M} &\longrightarrow \left(T\wt M \oplus T^*\wt M\right)_{T^*\wt M}.
\end{split}
\end{equation}
On either of the tangent or cotangent bundle, there are several equivalent choices of hyperkähler structure that we may make.
In \cite[Equation 6 (p.\ 106)]{CST}, the cotangent bundle $T^*\wt M$ was used, on which the hyperkähler structure was chosen to be
\begin{equation}\label{eq:cotangent_HK_structure}
\widehat{g}=
\begin{pmatrix}
\widetilde{g} & 0 \\ 0 & \widetilde{g}^{-1}
\end{pmatrix},
\quad
I_1=
\begin{pmatrix}
I & 0 \\ 0 & I^*
\end{pmatrix},
\quad
I_2=
\begin{pmatrix}
0 & -\widetilde{\omega}^{-1} \\ \widetilde{\omega} & 0
\end{pmatrix},
\quad
I_3=I_1I_2.
\end{equation}
We will instead be working with the corresponding hyperkähler structure on the tangent bundle $T\wt M$, that is
\begin{equation}\label{eq:HK_structure_TTMtilde}
	\widehat{g}=
	\begin{pmatrix}
	\widetilde{g} & 0 \\ 0 & \widetilde{g}
	\end{pmatrix},
	\qquad
	I_1=
	\begin{pmatrix}
	I & 0 \\ 0 & -I
	\end{pmatrix},
	\qquad
	I_2=
	\begin{pmatrix}
	0 & -\id \\ \id & 0
	\end{pmatrix},
	\qquad
	I_3=I_1I_2.
	\end{equation}
In fact, it will be convenient to make use of the identification of $T\widetilde M$ with $E_{\pi}$, cf.\ Remark \ref{rmk:TMtilde=[E]}, which allows us to identfy the vertical subbundle $\mathcal V \subseteq TT\wt M$ with the pullback of $E$ to $E_{\pi}$ by \eqref{diag:exact}. Then, we can define the metric on the total space of $E_{\pi}$ to be $\widetilde{g}$ on the horizontal subbundle $\mathcal H$, and to be the real part of the Griffiths hermitian form $h_{\mathrm G}$ on the vertical subbundle $\mathcal V$.
We can write this in terms of the tautological section $\Phi_{E}$, using $\nabla\Phi_{E}\colon T(E_\pi)\to E_\pi$ as the projection to the vertical part (cf.\ \S \ref{ssec:Connections}).
\begin{equation}
\begin{split}
\widehat{g}
&=\widetilde{g}
+\Re(h_{\mathrm{G}})(\nabla\Phi_{E},\nabla\Phi_{E})\\
&=r^2g_M -dr^2-r^2\widetilde{\varphi}^2+(-g_L+g_L\otimes g_M-\omega_L\otimes \omega_M)(\nabla\Phi_{E},\nabla\Phi_{E}).
\end{split}
\end{equation}
The hyperk\"ahler complex structures $I_1, I_2,I_3$ in \eqref{eq:HK_structure_TTMtilde} are then given by
\begin{equation}\label{eq:HK_structure_pi-1E}
\begin{split}
I_1 X
&=\bigg(I\left(p^{E_{\pi}}_{\widetilde{M}}\right)_*X\bigg)^{\nabla}-\mathrm{vert}( I_{\mathrm{G}}\nabla_X \Phi_{E}),\\
I_2 X&=\mathrm{vert}\circ\phi \bigg(\!\left(p^{E_{\pi}}_{\widetilde{M}}\right)_*X\bigg) -\bigg(\phi^{-1}\left(\nabla_X\Phi_{E}\right)\bigg)^{\nabla},\\
I_3X&=I_1I_2X.
\end{split}
\end{equation}

\subsection{Supergravity c-map}
\label{ssec:supergravity_c-map}
The supergravity c-map assigns to a PSK manifold $M$ a family of quaternionic Kähler manifolds $N_{2k}$ parametrised by an integer $k$ (we write $2k$ instead of $k$ for consistency with the notation in \cite{CST}).
The $k=0$ case will be called the \emph{undeformed c-map}, while the $k>0$ case will be called the \emph{deformed c-map}.
One can also define the c-map for $k<0$, but only the non-negative case is complete (cf.\ \cite[Theorem 13]{CortesCompProj} and \cite[Remark 9 (p.\ 287)]{CortesHKQK}).
Therefore, in this section, we will present an intrinsic formulation of the c-map where $k$ is a non-negative integer.
For our description of the c-map, we will adopt the twist approach of \cite{MaciaSwann2015}.

In order to introduce the supergravity c-map metric, we first need to define the following two tensor fields, given in terms of the $\nabla$-horizontal lift of $-I\xi$, which we call $Z$:
\begin{align}
g_{\Hq Z}\colonequals \widehat{g}|_{\spn{Z,I_1 Z,I_2 Z,I_3 Z}},
&\hfill&
g_{\perp}\colonequals \widehat{g}|_{\spn{Z,I_1 Z,I_2 Z,I_3 Z}^{\perp}}.
\end{align}

\begin{lemma}
	In terms of the VPHS data, we have
	\begin{equation}
	\begin{split}
	g_{\Hq Z}&=-\dif r ^2-r^2\widetilde{\varphi}^2-\Re{h_L}(\nabla\Phi_{E},\nabla\Phi_{E}),\\
	g_{\perp}&=r^2 g_M+\Re(h_L\otimes h_M)(\nabla\Phi_{E},\nabla\Phi_{E}).
	\end{split}
	\end{equation}
\end{lemma}
\begin{proof}
	Using \eqref{eq:HK_structure_pi-1E}, we can explicitly describe $\Hq Z\colonequals \langle Z,I_1Z,I_2Z,I_3Z \rangle$, via
	\begin{align}
	\begin{split}
	Z&=-(I\xi)^{\nabla},\\
	I_1Z&=\xi^{\nabla}-\vt(I_{\mathrm G}\nabla_{-I\xi^{\nabla}} \Phi_{E})
	=\xi^{\nabla},\\
	I_2Z&=-\vt \circ \phi(I\xi)+0
	=\vt(-\I\Phi_L+\I\ol{\Phi}_L),\\
	I_3Z&=I_1\vt(-\I\Phi_L+\I\ol{\Phi}_L)
	=-\vt(\Phi_L+\ol{\Phi}_L).
	\end{split}
	\end{align}
	We can see that $Z$ and $I_1 Z$ span the vertical part of $\widetilde{M}\to M$, whereas $I_2 Z$ and $I_3 Z$ span $\llbracket L \rrbracket \subseteq E$.
	The description of $g_{\Hq Z}$ and $g_{\perp}$ follows.
\end{proof}

In order to define the supergravity c-map, we make use of its description in terms of the twist.
Let $c\in\R_{\ge 0}$, we define $\widetilde{M}_{>c}\colonequals \{u\in\widetilde{M}\mid -\widetilde{g}_u(\xi,\xi)>c\}$.
\begin{prop}[\cite{MaciaSwann2015}, Lemma 5.2 (p.\ 1347)]
\label{prop:cmap_as_twist}
	The supergravity c-map metric is the twist of a constant scalar multiple of the tensor field
	\begin{equation}\label{eq:gH_Macia_Swann}
	g_{\mathrm{H}}=\frac{1}{r^2-2k}\left(g_{\perp}-\frac{r^2+2k}{r^2-2k}\,g_{\Hq Z}\right).
	\end{equation}
	with respect to the twist data
	\begin{equation}\label{eq:twist_data_c-map}
	(Z,\omega_{\mathrm{H}},f_{\mathrm{H}})=
	\left(-(I\xi)^{\nabla},
	\begin{pmatrix}
	-\widetilde{\omega}&0\\
	0&-\widetilde{\omega}
	\end{pmatrix},-\frac{1}{2}(r^2+2k)\right).
	\end{equation}
\end{prop}
Following \cite{MaciaSwann2015} and \cite{CST}, we will refer to $g_{\mathrm{H}}$ as the \emph{elementary deformation} of the hyperk\"ahler metric $\widehat{g}$.
\begin{rmk}
	Note that $\omega_{\mathrm{H}}$ is given in terms of the tautological section $\Phi_{E}$ by
	\begin{equation}
	\omega_{\mathrm{H}}
	= - \widetilde \omega - Q(\nabla \Phi_{E }, \nabla \Phi_{E })
	= - \widetilde \omega - \frac{1}{2}\,\dif(Q(\Phi_{E }, \nabla \Phi_{E})).
	\end{equation}
	Here, the second step follows from the fact that $\nabla$ is flat and preserves $Q$, so
	\begin{equation}
	\begin{split}
	&\dif(Q(\Phi_{E}, \nabla \Phi_{E}))(X,Y)\\ 
	&\qquad\qquad= \nabla_X(Q(\Phi_{E}, \nabla_Y \Phi_{E})) - \nabla_Y(Q(\Phi_{E}, \nabla_X \Phi_{E})) \\
	&\qquad\qquad\quad- Q(\Phi_{E}, \nabla_{[X,Y]} \Phi_{E})\\
	&\qquad\qquad=Q(\nabla_X \Phi_{E}, \nabla_Y \Phi_{E})-Q(\nabla_Y\Phi_{E}, \nabla_X \Phi_{E})\\
	&\qquad\qquad\quad - Q(\Phi_{E},([\nabla_X, \nabla_Y] -\nabla_{[X,Y]}) \Phi_{E})\\
	&\qquad\qquad=2Q(\nabla_X \Phi_{E}, \nabla_Y \Phi_{E}).
	\end{split}
	\end{equation}
\end{rmk}
In order to apply the twist construction to $T\wt M_{>2k}$, we would like to present it as a principal $\U(1)$-bundle so that we can apply Lemma \ref{lemma:twist}. To this end, we consider the following diagram.
\begin{equation}\label{diag:cube1}
\begin{tikzcd}[row sep=scriptsize, column sep=scriptsize]
&\widetilde{M}_{>2k} \arrow[rd] \arrow[dd,"{(\pi,-\widetilde{g}(\xi,\xi))}\quad"description]\arrow[dddr,"\pi",crossing over]&\\
T\widetilde{M}_{>2k} \arrow[ur]\arrow[dd,->, dashed]&& S \arrow[->,dd]\\
& M\times \R_{>2k} \arrow[dr,"\mathrm{pr}_M"']&\\
 E\times \R_{>2k}\arrow[dr] \arrow[ur] & & M \\
& E\arrow[uuul,<-,crossing over]\arrow[ur,"\pi_E"']
\end{tikzcd}
\end{equation}
The diagonal square in \eqref{diag:cube1} is a pullback, as we identified $E_{\pi}$ with $T\widetilde{M}$, and $T\widetilde{M}_{>2k}$ is its restriction to $\widetilde{M}_{>2k}\subseteq \widetilde{M}$.
The bottom square is also a pullback, and all of the undashed arrows in diagram \eqref{diag:cube1} commute.
So, by the universal property of pullbacks, we infer the existence of a (unique) dashed arrow, which makes the entire diagram commute.
Moreover, since the diagonal and bottom squares are pullbacks, it follows by categorical arguments that the diagram on the left side is a pullback as well.
Notice that the right diagram is also a pullback (in the category of smooth manifolds), as $\widetilde{M}_{>2k}=S\times \R_{>2k}$.
Consider now the following diagram.
\begin{equation}\label{diag:bold-square}
\begin{tikzcd}[row sep=scriptsize, column sep=scriptsize]
&\widetilde{M}_{>2k} \arrow[rd] \arrow[dd,"{(\pi,-\widetilde{g}(\xi,\xi))}" near end]&\\
T\widetilde{M}_{>2k} \arrow[rr, crossing over,->,very thick] \arrow[dr]\arrow[ur]\arrow[dd,->, very thick]&& S \arrow[->,dd,very thick]\\
& M\times \R_{>2k} \arrow[dr,"\mathrm{pr}_M"]&\\
E\times \R_{>2k}\arrow[ur] \arrow[dr]\arrow[rr,->,"\varpi",very thick] & & M \\
& E\arrow[ur,"\pi_E"']
\end{tikzcd}
\end{equation}
The highlighted square is the composition of the left and right ones, and therefore, it is itself a pullback square.
It follows that the map $T\widetilde{M}_{>2k}\to E\times \R_{>2k}$ is a principal $\U(1)$-bundle, since it is the pullback of the principal $\U(1)$-bundle $S\to M$.

Now we have all the ingredients to prove the main theorem of this paper.
\begin{theo}\label{th:sugra-c}
	Let $k\in\Z$ be non-negative, and let $\pi\colon \widetilde{M}\to M$ be a projective special K\"ahler manifold, whose associated $\U(1)$-bundle is $S$ and has principal connection $\wt\varphi$.
	Let $S_k\to M$ be the $\Z_k$-quotient of $S$ when $k>0$, and let $S_0=M\times\U(1)\to M$.
	Let $\varphi_k$ be the principal connection induced on $S_k\to M$ by $-\wt\varphi$ (cf.\ \eqref{eq:fundamentalVF_principalConn_discrete_quotient}, \eqref{eq:objects_S0}).
	The supergravity c-map with deformation parameter $k$ is $N_{2k}$, obtained as the highlighted pullback in the diagram below.
	\begin{equation}
	\begin{tikzcd}[row sep=scriptsize, column sep=scriptsize]
	N_{2k} \arrow[rr, crossing over,->,very thick] \arrow[dr]\arrow[dd,->, very thick]&& S_k \arrow[->,dd,very thick]\\
	& M\times \R_{>2k} \arrow[dr,"\mathrm{pr}_M"]&\\
	E\times \R_{>2k}\arrow[ur] \arrow[dr]\arrow[rr,->,"\varpi",very thick] & & M \\
	& E\arrow[ur,"\pi_E"']
	\end{tikzcd}
	\end{equation}
	The supergravity c-map metric on $N_{2k}$ is a constant multiple of
	\begin{align}\label{eq:sugra-c}
	\begin{split}
	g_{2k}
	&=\frac{r^2}{r^2-2k}\, g_M
	+\frac{1}{r^2-2k}\Re(h_L\otimes h_M)(\nabla\Phi_{E},\nabla\Phi_{E})\\
	&\quad+\frac{r^2+2k}{(r^2-2k)^2}\,\dif r ^2
	+\frac{4r^2}{(r^2+2k)(r^2-2k)^2}\Big(\varphi_{k}
	-\frac{1}{2}\, \pi_{k}^*( Q(\Phi_{E},\nabla \Phi_{E}))\Big)^2\\
	&\quad + \frac{r^2+2k}{(r^2-2k)^2}\Re{h_L}(\nabla\Phi_{E},\nabla\Phi_{E}).
	\end{split}
	\end{align}
\end{theo}
\begin{proof}
	In Proposition \ref{prop:cmap_as_twist}, we mentioned that the c-map can be interpreted as the twist of the elementary deformation of the hyperk\"ahler metric, i.e.\ of
	\begin{equation}
	\begin{split}
	g_{\mathrm{H}}
	&=\frac{1}{r^2-2k}\bigg(r^2 g_M+\Re(h_L\otimes h_M)(\nabla\Phi_{E},\nabla\Phi_{E})\\
	&\quad+\frac{r^2+2k}{r^2-2k}(\dif r ^2+r^2\widetilde{\varphi}^2+\Re{h_L}(\nabla\Phi_{E},\nabla\Phi_{E}))\bigg),
	\end{split}
	\end{equation}
	up to scaling.
	The twist data is $(Z,\omega_{\mathrm{H}},f_{\mathrm{H}})$.
	Notice that
	\begin{equation}
	\begin{split}
	\omega_{\mathrm{H}}
	&=-\widetilde \omega - \frac{1}{2}\,\dif(Q(\Phi_{E}, \nabla \Phi_{E}))\\
	&=\dif\Big((f_{\mathrm{H}}+k)\Big(\underbrace{\vphantom{\frac{a}{b}}-\widetilde{\varphi}}_{=:\varphi}\Big)\!\Big)+\dif\Big(\underbrace{-\frac{1}{2}\, Q(\Phi_{E},\nabla \Phi_{E})}_{=:\beta}\!\Big).
	\end{split}
	\end{equation}
	Furthermore, $\beta$ descends to $E$, as $\beta(Z)=0$ and
	\begin{equation}
	\Lie{Z} \beta
	=\iota_Z\dif \beta+\dif \iota_Z\beta
	=2\iota_Z Q(\nabla\Phi_{E},\nabla\Phi_{E})+0
	=0.
	\end{equation}
	Thus, $\omega_{\mathrm{H}}$ is of the form \eqref{eq:twist_data}, and we can apply Lemma \ref{lemma:twist} to regard $N_{2k}$ as the $\Z_k$-quotient of $T\wt M_{>2k}$ for $k>0$ and as $E\times \R_{>0}\times \U(1)$ for $k=0$.
	Using Lemma \ref{lemma:twist} we can also explicitly twist $g_{2k}$, obtaining
	\begin{align}
	\begin{split}
	g_{2k}
	&=\mathrm{tw}(g_{\mathrm{H}})\\
	&=\frac{1}{r^2-2k}\Bigg(r^2 g_M+\Re(h_L\otimes h_M)(\nabla\Phi_{E},\nabla\Phi_{E})\vphantom{\frac{1}{1}}+\frac{r^2+2k}{r^2-2k}\bigg(\dif r ^2\\
	&\quad+\frac{r^2}{f_{\mathrm{H}}^2}\Big(\varphi_{k} -\frac{1}{2}\, \pi_{k}^*( Q(\Phi_{E},\nabla \Phi_{E}))\Big)^2+\Re{h_L}(\nabla\Phi_{E},\nabla\Phi_{E})\bigg)\Bigg).
	\end{split}
	\end{align}
	Equation \eqref{eq:sugra-c} follows once we substitute $f_{\mathrm{H}}=-\frac{1}{2}(r^2+2k)$.
\end{proof}
\begin{rmk}
If we use the Griffiths hermitian form $h_{\mathrm{G}}=-h_L+h_L\otimes h_M$, since $Q=\Im h_{\mathrm{G}}$, we can rearrange the terms of \eqref{eq:sugra-c} to obtain
\begin{align}\label{eq:Grif_c-map}
\begin{split}
	g_{2k}
	&=\frac{r^2+2k}{(r^2-2k)^2}\,\dif r ^2
	+\frac{r^2}{r^2-2k}\, g_M
	+\frac{1}{r^2-2k}\Re h_{\mathrm{G}}(\nabla\Phi_{E},\nabla\Phi_{E})\\
	&\quad	+\frac{4r^2}{(r^2+2k)(r^2-2k)^2}\Big(\varphi_{k}
	-\frac{1}{2}\, \pi_{k}^*( \Im h_{\mathrm{G}}(\Phi_{E},\nabla \Phi_{E}))\Big)^2\\
	&\quad + \frac{2r^2}{(r^2-2k)^2}\Re{h_L}(\nabla\Phi_{E},\nabla\Phi_{E}).
\end{split}
\end{align}
We can do the same with the Weil hermitian form $h_{\mathrm{W}}=\ol h_L+h_L\otimes h_M$.
Since we again have $Q=\Im h_{\mathrm{W}}$, we get
\begin{align}
\begin{split}
	g_{2k}
	&=\frac{r^2+2k}{(r^2-2k)^2}\,\dif r ^2
	+\frac{r^2}{r^2-2k}\, g_M
	+\frac{1}{r^2-2k}\Re h_{\mathrm{W}}(\nabla\Phi_{E},\nabla\Phi_{E})\\
	&\quad	+\frac{4r^2}{(r^2+2k)(r^2-2k)^2}\Big(\varphi_{k}
	-\frac{1}{2}\, \pi_{k}^*( \Im h_{\mathrm{W}}(\Phi_{E},\nabla \Phi_{E}))\Big)^2\\
	&\quad + \frac{4k}{(r^2-2k)^2}\Re{h_L}(\nabla\Phi_{E},\nabla\Phi_{E}).
\end{split}
\end{align}
This last formulation of the c-map metric echoes the one presented in \cite[Equation (7) (p.\ 90)]{CortesCompProj}, once we set $\rho^{\mathrm{CDS}}=r^2-2k$.
\end{rmk}
\begin{eg}
	We revisit Example \ref{eg:CHn} and construct the supergravity c-map metric associated to $\Hy_\C^n$. We begin by defining a system of coordinates on $N_{2k}$.
	
	We already have global complex coordinates $X=(X_i)_{i=1}^n\colon M\to \C^n$, and a real coordinate $r>\sqrt{2k}$ on $\R_{>2k}$.
	We can get global coordinates $(w_i)_{i=0}^n$ for the fibres of $E$ by writing the tautological section $\Phi_{E}$ as
	\begin{equation}
		\Phi_{E}= \sum_{i=0}^n(w_is_i + \ol w_i \ol s_i),
	\end{equation}
	where $\{s_i,\ol s_i\}_{i=0}^n$ is the basis of the $\nabla$-parallel sections described in Example \ref{eg:CHn}.
	In particular, we have
	\begin{equation}
		\nabla \Phi_{E} =\sum_{i=0}^n(\dif w_i \otimes s_i + \dif \ol w_i \otimes \ol s_i).
	\end{equation}
	 Finally, let $t$ be a (local) coordinate function factoring through the $\Z_k$-quotient $q:S\rightarrow S_k$ such that $e^{-\I t}=z_0^k/|z_0|^k$.
	This gives on $S_k$ a global $1$-form
	\begin{equation}
		\dif t
		=-k\,\im\bigg(\frac{\dif z_0}{z_0}\bigg).
	\end{equation}
	Recall that on the base $M$, we have
	\begin{equation}
		\begin{split}
		g_M &=\frac{\norm{\dif X}^2}{1- \norm{X}^2} + \frac{\left|\sum_{i=1}^{n} \ol X_i\dif X_i\right|^2}{\left(1- \norm{X}^2 \right)^2},\\
		\omega_M &=\frac{1}{2\I}\Bigg(\frac{\sum_{j=1}^{n}\dif\ol X_j\wedge \dif X_j}{1- \norm{X}^2} + \frac{\sum_{i,j=1}^{n} X_i\dif\ol X_i\wedge \ol X_j\dif X_j }{\left(1- \norm{X}^2\right)^2}\Bigg)\\
		&=\frac{1}{2}\,\dif\,\im\Bigg(\sum_{i=1}^{n}\frac{\ol X_i\dif X_i}{1- \norm{X}^2}\Bigg).
		\end{split}
	\end{equation}
	On $S$, we have the principal connection $1$-form
	\begin{equation}
		\wt\varphi = \Im\left(\frac{\dif z_0}{z_0} - \sum_{i=1}^{n}\frac{\ol X_i\dif X_i}{1- \norm{X}^2}\right).
	\end{equation}
	Following the construction in Theorem \ref{th:sugra-c}, $\varphi=-\wt\varphi$ induces on $S_k$ the principal connection $\varphi_k$ (cf.\ \S \ref{sec:twist}).
	By \eqref{eq:fundamentalVF_principalConn_discrete_quotient}, we have
	\begin{equation}
		\begin{split}
			q^*\varphi_k
			&=-k\wt\varphi
			= \dif t + \frac{k}{1- \norm{X}^2}\,\im\Bigg(\sum_{i=1}^{n}\ol X_i\dif X_i\Bigg).
		\end{split}
	\end{equation}
	Next, we have the vertical forms
	\begin{equation}
		\begin{split}
		h_{\mathrm{G}}(\Phi_{E}, \nabla\Phi_{E})
		&=\sum_{i=1}^{n} \ol w_i\dif w_i-\ol w_0 \dif w_0,\\
		h_{\mathrm{G}}(\nabla\Phi_{E}, \nabla\Phi_{E})
		&=\sum_{i=1}^{n} \dif \ol w_i\otimes \dif w_i-\dif \ol w_0\otimes \dif w_0,\\
		h_L(\nabla\Phi_{E},\nabla\Phi_{E})
		&=\frac{\left(\dif\ol w_0-\sum_{i=1}^{n}X_i\dif\ol w_i\right)\otimes\left(\dif w_0-\sum_{j=1}^{n}\ol X_j\dif w_j\right)}{1- \norm{X}^2}.
		\end{split}
	\end{equation}
	Putting all of the above together according to \eqref{eq:Grif_c-map}, we obtain
	\begin{align}
	\begin{split}
	g_{2k}
	&=\frac{r^2+2k}{(r^2-2k)^2}\,\dif r ^2
	+\frac{r^2}{r^2-2k}\, \Bigg(\frac{\norm{\dif X}^2}{1- \norm{X}^2} + \frac{|\sum_{i=1}^{n} \ol X_i\dif X_i|^2}{\left(1- \norm{X}^2 \right)^2}\Bigg)\\
	&\quad+\frac{1}{r^2-2k}\left(\sum_{i=1}^{n} |\dif w_i |^2-|\dif w_0|^2 \right)
	+ \frac{2r^2}{(r^2-2k)^2}\frac{|\dif w_0-\sum_{j=1}^{n}\ol X_j\dif w_j|^2}{1- \norm{X}^2}\\
	&\quad	+\frac{4r^2}{(r^2+2k)(r^2-2k)^2}\Bigg(\dif t + \frac{k}{1- \norm{X}^2}\,\im\Bigg(\sum_{i=1}^{n}\ol X_i\dif X_i\Bigg)\\
	&\quad-\frac{1}{2}\,\Im \Bigg(\sum_{i=1}^{n} \ol w_i\dif w_i-\ol w_0 \dif w_0 \Bigg)\!\Bigg)^2.
	\end{split}
	\end{align}
	This is the deformed Ferrara--Sabharwal metric appearing in \cite[Corollary 15 (p.\ 97)]{CortesCompProj} upon taking
	\begin{align}
	\rho^{\mathrm{CDS}}=r^2-2k,
	\qquad w^{\mathrm{CDS}}_i=\frac{(-1)^{\delta_{i,0}}}{\sqrt{2}}\,\ol w_i,
	\qquad t^{\mathrm{CDS}}=4t.
	\end{align}		
	 This was originally derived by physicists in \cite{CFG1989, FerSab, RLSV} as a string-theoretic moduli space metric with perturbative quantum corrections.
	
	When $k=0$, this is a left invariant metric on (a $\Z$-quotient of) the non-compact Wolf space $\SU(n+1,2)/\mathrm{S}(\U(n+1)\times \U(2))$.
\end{eg}
\subsection{Heisenberg action}

While there is always a local basis of $\nabla$-parallel sections of $E$, a global basis may fail to exist because of the presence of monodromy.
We can however get around this issue by considering the universal cover of $M$ and pulling back $E$ along the covering map.
This gives us a VPHS whose Gauß--Manin connection has no monodromy.
We will show that in the absence of monodromy, the supergravity c-map admits the action of the Heisenberg group.

In particular, we would like to obtain the Killing vector fields on $N_{2k}$ that generate the Heisenberg action.
In order to do this, we begin by constructing vertical Killing vector fields on the \emph{rigid} c-map space.
The following proposition will show that the bundle map $\vt\colon E_{\pi}\to T(E_{\pi})$ maps pullbacks of parallel sections of $E$ to Killing vector fields of $E_{\pi}$.
\begin{prop}\label{prop:vertical_Killing}
	Let $s\in\Gamma(E)$ be such that $\nabla s=0$.
	Then, the vertical vector field $v_s\in\Gamma(T(E_{\pi}))$ corresponding to $s$ is Killing and $\omega_{\mathrm{H}}$-Hamiltonian with Hamiltonian function
	\begin{equation}
	f_s=Q(s,\Phi_{E}).
	\end{equation}
	Moreover, it preserves $Z$, $f_{\mathrm{H}}$, $I_j$, and $\widehat\omega_j=\widehat{g}(I_j\cdot,\cdot)$, for $j=1,2,3$ as defined in \eqref{eq:twist_data_c-map} and \eqref{eq:HK_structure_pi-1E}.
\end{prop}
\begin{proof}
	First of all, notice that since $s$ is $\nabla$-parallel, $\nabla(\nabla_{v_s}\Phi_{E})=\nabla s=0$.
	We now compute the Lie derivative of $\widehat{g}$ along $v_{s}$, obtaining
	\begin{align}
	\begin{split}
	&\Lie{v_{s}}\widehat{g}(X,Y)
	=\Lie{v_{s}}\big(\widetilde{g}
	+\Re h_{\mathrm G}(\nabla\Phi_{E},\nabla\Phi_{E})\big)(X,Y)\\
	&\qquad=\Lie{v_{s}}\big(\Re h_{\mathrm G}(\nabla\Phi_{E},\nabla\Phi_{E})\big)(X,Y)\\
	&\qquad=v_{s}\left(\Re h_{\mathrm G} (\nabla_{X}\Phi_{E},\nabla_{Y}\Phi_{E})\right)\\
	&\qquad\quad-\Re h_{\mathrm G} (\nabla_{[v_{s},X]}\Phi_{E},\nabla_{Y}\Phi_{E})
	-\Re h_{\mathrm G} (\nabla_{X}\Phi_{E},\nabla_{[v_{s},Y]}\Phi_{E})\\
	&\qquad=\Re (\nabla_{v_{s}} h_{\mathrm G}) (\nabla_{X}\Phi_{E},\nabla_{Y}\Phi_{E})\\
	&\qquad\quad+\Re h_{\mathrm G} \left([\nabla_{v_{s}},\nabla_{X}]\Phi_{E}-\nabla_{[v_{s},X]}\Phi_{E},\nabla_{Y}\Phi_{E}\right)\\
	&\qquad\quad+\Re h_{\mathrm G}\left([\nabla_{v_{s}},\nabla_{Y}]\Phi_{E}-\nabla_{[v_{s},Y]}\Phi_{E},\nabla_{X}\Phi_{E}\right)\\
	&\qquad=\Re(\nabla_{v_{s}} h_{\mathrm G})(\nabla_{X}\Phi_{E},\nabla_{Y}\Phi_{E}).
	\end{split}
	\end{align}
	The last expression vanishes because $\nabla_{v_s} h_G$ is to be interpreted as $\nabla_{(\pi_{E})_*v_s}h_{G}$ by the definition of a pullback connection, and $v_s$ is a vertical tangent vector on the total space of the bundle $\pi_E\colon E\to M$.
	
	We can now check that $f_s$ is a Hamiltonian function of $v_{s}$ by computing
	\begin{align}
	\dif f_s
	&=\nabla (Q(s,\Phi_{E}))
	=Q(s,\nabla\Phi_{E})
	=Q(\nabla_{v_{s}}\Phi_{E},\nabla\Phi_{E})
	=-\iota_{v_{s}}\omega_{\mathrm{H}}.
	\end{align}
	
	The function $r$ is defined on the base, and therefore it is invariant along vertical vector fields.
	In particular $\Lie{v_{s}}f_{\mathrm{H}}=0$.
	
	The vector field $Z$ is also invariant, as $\widetilde{\varphi}=r^{-2}\widehat{g}(Z,\cdot)$ is
	\begin{align}
	\Lie{v_{s}}\widetilde{\varphi}
	&=\dif \iota_{v_{s}}\wt \varphi+\iota_{v_{s}}\dif \wt\varphi
	=0-2\iota_{v_{s}}\widetilde{\omega}
	=0.
	\end{align}

	The invariance of $I_k$ is equivalent to that of $\widehat\omega_k$.
	Note that $\widehat\omega_1=2\widetilde{\omega}+\omega_{\mathrm{H}}$ is a sum of invariant $2$-forms.
	Their invariance follows from the observation that $\widetilde{\omega}$ is horizontal, and $v_{s}$ is $\omega_{\mathrm{H}}$-Hamiltonian with $\omega_{\mathrm{H}}$ closed.
	
	Meanwhile, for $\widehat\omega_2$, we compute
	\begin{align}
	\begin{split}
	\Lie{v_{s}}\widehat\omega_2
	&=\dif\iota_{v_{s}}\widehat\omega_2
	=\dif\big(\widehat{g}(I_2v_{s},\cdot)\big)
	=-\dif\Big(\widehat{g}\big((\phi^{-1}(s))^{\nabla},\cdot\big)\Big)\\
	&=-\dif\Big(\widetilde{g}\big(\phi^{-1}(s),\cdot\big)\Big)
	=-2\Alt\Big(\nabla^{\widetilde{g}}\widetilde{g}\big(\phi^{-1}(s),\cdot\big)\Big)\\
	&=-2\Alt\Big(\widetilde{g}\big(\nabla^{\widetilde{g}}(\phi^{-1}(s)),\cdot\big)\Big)
	=4\Alt\Big(\widetilde{g}\big(\Re \wt\eta_{\cdot}(\phi^{-1}(s)),\cdot\big)\Big).
	\end{split}
	\end{align}
	This quantity vanishes because $\Re\wt\eta$ is symmetric.
	The invariance of $I_3=I_1 I_2$ follows.
\end{proof}

Killing vector fields on a hyperk\"ahler manifold satisfying the properties resulting from Proposition \ref{prop:vertical_Killing} can be used to obtain Killing vector fields on the corresponding quaternionic K\"aher manifold via the following proposition.
\begin{prop}[\cite{CST} Proposition 3.5 (p.\ 108)]\label{prop:Killing-twist}
	Let $X$ be a Killing vector field on $(E_{\pi},\widehat{g})$ which is $\omega_{\mathrm H}$-Hamiltonian with Hamiltonian function $f_X$, and which preserves the hyperk\"ahler structure $I_1$, $I_2$, $I_3$ and the twist data $Z$, $f_{\mathrm{H}}$.
	Then, the vector field
	\begin{equation}\label{eq:twisted_killing}
	X^{\mathrm{Q}}\colonequals \mathrm{tw}\left(X-\frac{f_X}{f_{\mathrm{H}}}Z\right)
	\end{equation}
	is Killing with respect to $g_{2k}$. 
\end{prop}
Let $s$ thus be a section of $E$ such that $\nabla s=0$, and let $v_s$ be the corresponding vertical vector field on $E_{\pi}$.
We obtain a Killing vector field
\begin{equation}
w_s\colonequals \mathrm{tw}\left(v_s-\frac{f_s}{f_{\mathrm{H}}}Z\right).
\end{equation}
Notice that $-\wt\varphi(v_s)=0$, so by $Z$-invariance, $v_s$ is the $(-\wt\varphi)$-horizontal lift of some vector field on $\R_{>2k}\times M$.
Let $v'_s$ be the $\varphi_k$-horizontal lift of said vector field on $N_{2k}$.
The twist of $v_s$ is then $v'_s-\beta(v_s)Z_k$, and that of $Z$ is $-f_{\mathrm{H}}Z_k$, where $Z_k$ is the fundamental vector field of the $\U(1)$-bundle $N_{2k}\to E_{\mathrm{pr}_M}$.
Therefore,
\begin{align}
\begin{split}
w_s
&=v'_s-\beta(v_s)Z_k+f_{s}Z_k
=v'_s+\frac{1}{2}\, Q(\Phi_{E},\nabla_{v_s} \Phi_{E})Z_k+ Q(s, \Phi_{E})Z_k\\
&=v'_s+\frac{1}{2}\, Q(\Phi_{E},s)Z_k+ Q(s, \Phi_{E})Z_k
=v'_s+\frac{1}{2}\, Q(s,\Phi_{E})Z_k.
\end{split}
\end{align}

Now, in order to see the Heisenberg action, we make use of the following result relating commutators of Killing vector fields on the quaternionic Kähler manifold to those on the hyperkähler manifold.

\begin{prop}
Let $X_1,X_2$ be Killing vector fields on $(E_{\pi},\widehat{g})$ as in Proposition \ref{prop:Killing-twist} with Hamiltonian functions $f_{X_1}, f_{X_2}$.
Then the corresponding Killing vector fields $X_1^{\mathrm{Q}}$ and $X_2^{\mathrm{Q}}$ on $N_{2k}$ satisfy
	\begin{equation}
	[X_1^{\mathrm{Q}},X_2^{\mathrm{Q}}]= \tw([X_1,X_2]) + \tw(\omega_{\mathrm H}(X_1,X_2))Z_k.
	\end{equation}
\end{prop}
\begin{proof}
See \cite[Theorem 3.8 (p.\ 109)]{CST} combined with Proposition \ref{prop:Killing-twist}.
\end{proof}

\begin{cor}\label{cor:Heisenberg}
	The vertical Killing vector fields $w_s$ associated to $\nabla$-parallel sections $s$ of $E$ and the Killing vector field $Z_k$ satisfy the Heisenberg algebra relations
	\begin{align}
		[w_s, Z_k]=0,
		&\hfill &
		[w_{s_1},w_{s_2}] = Q(s_1,s_2)Z_k.
	\end{align}
\end{cor}

\begin{rmk}
	The vertical Killing vector field $w_s$ associated to a $\nabla$-parallel section $s$ of $E$ can be readily integrated to yield an isometry $\psi_s$ of $N_{2k}$, given by 
	\begin{equation}
		\psi_s(p,x_p,u_p)= \left(p,x_p+s_p,u_p\exp\bigg(\frac{1}{2}\,Q_p(s_p,x_p)\bigg)\right).
	\end{equation} 
	Here, $p$ is a point on $M \times \R_{>2k}$, while $x_p$ and $u_p$ are points in the fibres over $p$ of the respective pullbacks of $E\rightarrow M$ and $S_k\rightarrow M$ to $M \times \R_{>2k}$.
	
	We can now explicitly see that this gives an action of the Heisenberg group:
	\begin{equation}
		\begin{split}
		\psi_{s_1}\circ \psi_{s_2}(p,x_p,u_p)
		&=\psi_{s_1+s_2}\bigg(p,x_p,u_p\exp\bigg(\frac{1}{2}\,Q(s_1,s_2)_p\bigg)\bigg).
		\end{split}
	\end{equation} 
\end{rmk}

\subsection{Functoriality}\label{ssec:functoriality}

The supergravity c-map is a natural construction. This means that isomorphisms of PSK manifolds lift to isometries of quaternionic Kähler manifolds in a natural way. In this subsection, we describe this lifting of isomorphisms explicitly.

Consider the PSK isomorphism $(\wt\psi,\psi)$ from $(\pi\colon\widetilde{M}\rightarrow M,\wt g,I,\wt\omega,\wt\nabla,\xi)$ to $(\pi'\colon\widetilde{M'}\rightarrow M',\wt g',I',\wt\omega',\wt\nabla',\xi')$.
In the following discussion, any construction associated to the PSK manifold $M$ will have a primed counterpart associated to the PSK manifold $M'$.
Using the structure of $L\colonequals \wt M\times_{\C^{\times}} \C$ and $L'$ as associated fibre bundles and the isomorphism $\wt\psi\colon \wt M \rightarrow \wt M'$, we can produce a complex vector bundle isomorphism $\psi^L\colonequals \wt\psi\times_{\C^{\times}} \id_{\C}\colon L \rightarrow L'$ covering $\psi$.
This can be regarded as an extension of $\wt\psi$.
Using $\psi^L$, we can then construct the maps $\psi^{E_{\C}}\colon E_{\C} \rightarrow E'_{\C}$ and $\psi_{\mathrm{pr}}^{E}\colon E\times \R_{>2k} \rightarrow E'\times \R_{>2k}$ via
\begin{equation}
\begin{split}
\psi^{E_{\C}} &= \psi^L + \psi^L \otimes (\psi^{1,0}_*)_{\psi^{-1}} + \ol\psi^L \otimes (\psi^{0,1}_*)_{\psi^{-1}} + \ol \psi^L, \\
\psi_{\mathrm{pr}}^{E} &= \left.\psi^{E_{\C}}\right|_{E}\times\id_{\R_{>2k}}.
\end{split}
\end{equation}
Since $\wt \psi$ is compatible with the Kähler structures on $\wt M$ and $\wt M'$, the map $\psi^L$ is compatible with the Hermitian forms $h_L$ and $h_{L'}$, and with the Chern connections $\chi$ and $\chi'$ (whose imaginary parts are $\wt\varphi$ and $\wt\varphi'$ respectively).
Since $\wt\psi_*$ maps $\zeta$ to $\zeta'$, the map $\psi^L$ sends the unit circle bundle $S \subset L$ to $S'\subset L'$.
So, we obtain a map $\psi^S_k\colon S_k\rightarrow S'_k$ pulling back the principal $\U(1)$-connection $\varphi'_k$ to $\varphi_k$ via the codifferential $(\psi^S_k)^*$.

By Theorem \ref{th:sugra-c}, the quaternionic K\"ahler manifold $N_{2k}$ obtained from $M$ via the supergravity c-map is the fibred product of the bundles $E\times \R_{>2k}\rightarrow M$ and $S_k \rightarrow M$.
An analogous statement holds for $N'_{2k}$.
Therefore, the maps $\psi_{\mathrm{pr}}^{E}$ and $\psi^S_k$ (both of which cover $\psi$) canonically induce a map $\psi^N\colon N_{2k}\rightarrow N'_{2k}$ covering $\psi$.

As the map $\wt \psi$ relates $\wt \omega$ and $\wt \nabla$ to $\wt \omega'$ and $\wt \nabla'$, the induced map $\psi_{\mathrm{pr}}^{E}$ relates the (pullbacks of the) bilinear form $Q$ and the connection $\nabla$ on the bundle $E\times \R_{>2k} \rightarrow M\times \R_{>2k}$, to $Q'$ and $\nabla'$ on $E'\times \R_{>2k} \rightarrow M'\times \R_{>2k}$.
Similarly, the K\"ahler structure of $M$ and the coordinate $r$ are related to the corresponding counterparts on $M'$.
 Therefore, the map $\psi^N$ is an isometry between the quaternionic K\"ahler manifolds $(N_{2k},g_{2k})$ and $(N'_{2k},g'_{2k})$.
As a result of naturality, the assignment $(\psi,\wt\psi)\mapsto \psi^N$ is a functor.
This provides us with the desired functorial lift of PSK isomorphisms.

In particular, we can consider $1$-parameter subgroups of PSK automorphisms and take the differential of the action at the identity to get vector fields that act as infinitesimal automorphisms.
Lifts of infinitesimal PSK automorphisms to infinitesimal isometries (i.e.\ Killing vector fields) of the supergravity c-map metric were constructed in a different way in \cite[Theorem 3.15 (p.\ 112)]{CST}. Unsurprisingly, the above functorial construction reproduces the same result.

\begin{prop}\label{prop:vorsicht-funktor}
	Let $\wt M \rightarrow M$ be a PSK manifold and let $\wt X$ and $X$ be vector fields on $\wt M$ and $M$ respectively, which generate a $1$-paramater family of PSK automorphisms $(\wt \psi_t,\psi_t)$ on $\wt M \rightarrow M$. Then, the complete lift $\wt X^T$ on $T\wt M \cong E_{\pi}$ is $\omega_{\mathrm H}$-Hamiltonian with Hamiltonian function
	\begin{equation}
	f^T_X \colonequals f_{\wt X^T} = -\Big(f_{\mathrm{H}}+k\Big)\widetilde{\varphi}(X)-\frac{1}{2}\, Q(\Phi_{E},\nabla_X \Phi_{E}).
	\end{equation}
	Moreover, the functorial lift $\psi_t^N$ of $(\wt \psi_t,\psi_t)$ is generated by the Killing vector field $(\wt X^T)^{\mathrm Q}$ on $N_{2k}$ arising from $\wt X^T$ and the choice of $\omega_{\mathrm H}$-Hamiltonian function $ f^T_X$, in accordance with \eqref{eq:twisted_killing}.
\end{prop}
\begin{proof}
	In order to prove the first part, our strategy will be to use Lemma \ref{lem:tw_symmetries}. We already know $\omega_{\mathrm H}$ is non-degenerate. We need to show that $\wt X^T$, interpreted as vector field tangent to $E_\pi$, preserves
	\begin{align}\label{eq:to-show}
		f_{\mathrm H} = -\frac{1}{2}(r^2 + 2k),\qquad
		\varphi=-\wt\varphi,\qquad
		\beta = -\frac{1}{2}\,Q(\Phi_{E}, \nabla \Phi_{E}).
	\end{align}
	For this purpose, we will describe the flow of $\wt X^T$ on $E_\pi$ and show that it preserves these objects.
	
	Consider first a PSK automorphism $(\wt\psi,\psi)$.
	By the universal property of the pullback we get a map $\psi^{(E_{\C})_{\pi}}\colon (E_{\C})_{\pi}\to E_{\pi}$, which in turn restricts to a map $\psi^{E_{\pi}}\colon E_{\pi}\to E_{\pi}$, as $\psi^{E_{\C}}$ preserves the real structure.
	Given $Y\in T\wt M_{>2k}$, we determine the element of $E_{\pi}$ corresponding to $\psi_* Y$ to be
	\begin{equation}\label{eq:naturality_psi_E}
		\begin{split}
		\phi(\wt\psi_*Y)
		&= \chi(\wt\psi_*Y)\Phi_L
		+ \Phi_L \otimes \pi_*\wt\psi_*Y^{1,0}
		+ \ol\Phi_L \otimes \pi_*\wt\psi_*Y^{0,1}
		+ \ol \chi(\psi_*Y)\ol \Phi_L\\
		&=\chi(Y)\Phi_L
		+ \Phi_L \otimes \psi_*\pi_*Y^{1,0}
		+ \ol\Phi_L \otimes \psi_*\pi_*Y^{0,1}
		+ \ol \chi(Y)\ol \Phi_L.
		\end{split}
	\end{equation}
	Here, we used the fact that $\chi$ and $\ol \chi$ are invariant under the action of the diffeomorphism $\wt \psi$, and that $\wt \psi$ is a map of the bundle $\pi\colon \wt M \rightarrow M$ covering $\psi$.
	We know that the tautological section $\Phi_L$ is preserved by the map (induced by) $\psi^L$.
	Therefore, $\phi(\wt\psi_*Y)$ is actually equal to $\psi^{(E_{\C})_{\pi}}(\phi(Y))$, implying the commutativity relation $\psi^{E_{\pi}}\circ \phi=\phi\circ\wt\psi_*$.
	The vector field $\wt X^T$, interpreted in $E_{\pi}$, can be written as
	\begin{equation}
	\phi_*(\wt X^T{}_{Y})
	=\left.\frac{\dif}{\dif t}\phi\left((\wt \psi_t)_*Y\right)\right\vert_{t=0}
	=\left.\frac{\dif}{\dif t}(\psi_t)^{E_\pi}\left(\phi(Y)\right)\right\vert_{t=0}.
	\end{equation}
	It follows that $(\psi_t)^{E_{\pi}}$ is the $1$-parameter family of automorphisms induced by $\wt X^T$ in $E_{\pi}$.
	Since $\wt\psi_t$ preserves $r^2=-\wt g(\xi,\xi)$, $\wt\varphi=\Im(\chi)$, $\wt\omega$, and $\wt\nabla$, the induced lift $\psi_t^{E_\pi}$ preserves the objects in \eqref{eq:to-show}, which are then $\wt X^T$-invariant.
	
	By combining \eqref{eq:twist_as_quotient} and \eqref{eq:twisted_killing}, we deduce that $(\wt X^T)^{\mathrm{Q}}$ is the pushforward of $\wt X^T$ via the quotient map $T\wt M_{>2k}\to N_{2k}$.
	Therefore, in order to prove the last part, it is enough to show that given some PSK automorphism $(\wt\psi,\psi)$, the maps induced by $\wt\psi_*\colon T\wt M_{>2k}\to T\wt M_{>2k}$ on $E\times\R_{>2k}$ and $S$ via the diagram \eqref{diag:bold-square}, are $\psi^{E}_{\mathrm{pr}}$ and $\psi^S$ respectively.
	 From the fact that $\wt\psi$ preserves $r^2$, we infer that it restricts to $\psi^S$ on $S$ and it descends to $\psi\times\id_{\R_{>2k}}$ on $M\times \R_{>2k}$.
	Notice that in the diagram \eqref{diag:cube1}, for every object, we have defined an associated automorphism, and that they all commute with the undashed arrows in the diagram.
	It follows that they also commute with the dashed one, implying that $\wt\psi_*$, descends to $\psi^{E}_{\mathrm{pr}}$ on $E\times\R_{>2k}$.
\end{proof}
\begin{rmk}\label{rmk:k_integer}
The integrality of $k$ is crucial in our functorial lifting of general PSK isomorphisms.
For instance, if $k=\frac{1}{2}$, by Remark \ref{rmk:rational_twist}, the twist of $E_{\pi}$ can be identified with a double cover of itself, and in general, automorphisms cannot be lifted to a double cover in a canonical way.
In \cite{CST}, the integrality of $k$ was not assumed, but this is nevertheless consistent with the above, as infinitesimal automorphisms can still be lifted for general $k$, in a canonical way.
\end{rmk}

\begin{rmk}\label{rmk:semidirectHeis}
	The action of the automorphism group $\mathrm{Aut}_{\mathrm{PSK}}(M)$ of the PSK manifold $\wt M \rightarrow M$ on the space of parallel sections $s$ of $E$ induces an action on the Heisenberg isometries $\psi_s$.
	This action is compatible with the conjugation by the functorial lift of PSK automorphisms.
	Therefore, the subgroup of isometries generated by the functorial lifts of PSK automorphisms and the Heisenberg isometries is a semidirect product 
	\begin{equation}
		\mathrm{Aut}_{\mathrm{PSK}}(M)\ltimes \mathrm{Heis}_{\rk(E)+1} \le \mathrm{Isom}(N_{2k}).
	\end{equation}
	This was shown for $\mathbf{H}_\C^n$ in a different manner in \cite[\S 3]{CortesRoserThung}.
	For the $n=0$ case, it was shown in \cite[Theorem 4.5 (p.\ 116)]{CST} that this semidirect product is in fact the full group of isometries.
	An advantage of our functorial approach is that it makes the action of the semidirect product on $N_{2k}$ transparent and explicit in general.
\end{rmk}
\section*{}
\bibliography{Bibliography}
\bibliographystyle{alpha}
\end{document}